\documentclass[final,hidelinks,onefignum,onetabnum]{siamart190516}

\usepackage{lipsum}
\usepackage{amsfonts}
\usepackage{graphicx}
\usepackage{epstopdf}
\usepackage{algorithmic}
\ifpdf
  \DeclareGraphicsExtensions{.eps,.pdf,.png,.jpg}
\else
  \DeclareGraphicsExtensions{.eps}
\fi
\usepackage{enumitem}
\usepackage{amssymb}
\usepackage{mathtools}
\usepackage{autonum}

\renewcommand{\u}{\vec{u}}
\newcommand{\vpsi}{\vec{\psi}}

\newcommand{\x}{\vec{x}}
\newcommand{\xx}{\tilde{x}}
\newcommand{\rr}{\tilde{r}}
\newcommand{\G}{\vec{G}}
\newcommand{\y}{\vec{y}}
\newcommand{\V}{\vec{V}}
\newcommand{\w}{\vec{w}}
\renewcommand{\v}{\vec{v}}
\newcommand{\g}{\vec{g}}

\newcommand{\vchi}{\vec{\chi}}

\newcommand{\f}{\vec{f}}
\newcommand{\R}{\mathbb{R}}

\newcommand{\Th}{\mathcal{T}_h}

\newcommand{\lnhh}{\lvert\ln h\rvert}
\newcommand{\Om}{\Omega}

\DeclarePairedDelimiter{\abs}{\lvert}{\rvert}
\DeclarePairedDelimiter{\norm}{\lVert}{\rVert}

\DeclarePairedDelimiter{\onenorm}{\lVert}{\rVert_{L^1(\Omega)}}
\DeclarePairedDelimiter{\inftynorm}{\lVert}{\rVert_{L^{\infty}(\Omega)}}
\DeclarePairedDelimiter{\twonorm}{\lVert}{\rVert_{L^2(\Omega)}}
\DeclarePairedDelimiter{\stwonorm}{\lVert}{\rVert^2_{L^2(\Omega)}}

\newsiamremark{remark}{Remark}
\newsiamremark{hypothesis}{Hypothesis}
\crefname{hypothesis}{Hypothesis}{Hypotheses}
\newsiamthm{claim}{Claim}
\newsiamthm{assumption}{Assumption}
\crefname{assumption}{Assumption}{Assumptions}

\headers{Stokes global and local pointwise error estimates}{N. Behringer, D. Leykekhman and B. Vexler}

\title{Global and local pointwise error estimates for finite element approximations to the Stokes problem on convex polyhedra%
\thanks{\funding{The first author gratefully acknowledges support from the International Research	Training Group IGDK, funded by the German Science Foundation (DFG) and the Austrian Science Fund (FWF). The second author was supported by the NSF grant DMS-1913133.}}}

\author{Niklas Behringer\thanks{Chair of Optimal Control, Center for Mathematical Sciences, Technical University of Munich, 85748 Garching by Munich, Germany (\email{nbehring@ma.tum.de}, \email{vexler@ma.tum.de}).}
\and Dmitriy Leykekhman\thanks{Department of Mathematics, University of Connecticut, Storrs, CT 06269 (\email{dmitriy.leykekhman@uconn.edu}).}
\and Boris Vexler\footnotemark[2]}

\usepackage{amsopn}

\ifpdf
\hypersetup{
  pdftitle={Global and local pointwise error estimates for finite element approximations to the Stokes problem on convex polyhedra},
  pdfauthor={N. Behringer, D. Leykekhman, and B. Vexler}
}
\fi

\begin{document}

\maketitle

\begin{abstract}
  The main  goal of the paper is to show new stability and localization results for the finite element solution of the Stokes system in $W^{1,\infty}$ and $L^{\infty}$ norms under standard assumptions on the finite element spaces on quasi-uniform meshes in two and three dimensions.  Although interior error estimates are well-developed for the elliptic problem, they appear to be new for the Stokes system on unstructured meshes.  To obtain these results we extend previously known stability estimates for the Stokes system using regularized Green's functions. 
\end{abstract}



\section{Introduction}
\label{section:introduction}
In the introduction and the major part of the paper we focus on the three-dimensional setting. However, our results are valid in two dimensions and we comment on that at the end of the paper.
We assume $\Omega \subset \R^3$ is a convex polyhedral domain, on which we consider the following Stokes problem:
\begin{subequations}
	\begin{align}
	-\Delta \u + \nabla p &= \f \quad \text{ in } \Omega, \label{eq:def_stoke_1} \\
	\nabla \cdot \u &= 0 \quad \text{ in } \Omega, \label{eq:def_stoke_2}\\
	\u &= \vec 0 \quad \text{ on } \partial \Omega, \label{eq:def_stoke_3}
	\end{align}
\end{subequations}
with $\f =(f_1,f_2,f_3)$ be such that $\u\in (H^1_0(\Omega)\cap L^{\infty}(\Om))^3$ or respectively $\u\in (H^1_0(\Omega)\cap W^{1,\infty}(\Om))^3$ and $p\in L^{\infty}(\Omega)$.
The solution $p$ is unique up to a constant, we choose $p\in L^2_0(\Omega)$, i.e. $p$ has zero mean.

This paper is the first paper in our program to establish best approximation results for the fully discrete approximations for transient Stokes systems in  $L^\infty$ and $W^{1,\infty}$ norms. Similar program was carried out by the last two authors for the parabolic problems in a series of papers \cite{MR3470741,MR3498514,MR3606467,MR3687851}. The approach there relies on stability of the Ritz projection, resolvent estimates in  $L^\infty$ and $W^{1,\infty}$ norms and discrete maximum parabolic regularity. We intend to derive corresponding results for the Stokes systems. In this paper,
we give a new $L^{\infty}$ stability result of the form
\begin{equation}
\inftynorm{\u_h} \leq C\lnhh \Big(\lnhh\inftynorm{\u}+ h\inftynorm{p} \Big). \label{eq:linfty_global}
\end{equation}
In a second step we prove respective local versions of \cref{eq:linfty_global} and of the corresponding $W^{1,\infty}$ results from \cite{MR3422453,MR2945141}. These estimates take the form
\begin{multline}
\norm{\nabla \u_h}_{L^{\infty}(D_1)} + \norm{p_h}_{L^{\infty}(D_1)}
\\ \leq C\left(\norm{\nabla\u}_{L^{\infty}(D_2)} + \norm{p}_{L^{\infty}(D_2)} \right)
+C_d\Big(\twonorm{\nabla\u} + \twonorm{p}\Big)
\end{multline}
and
\begin{multline}
\norm{\u_h}_{L^{\infty}(D_1)}
\leq C\lnhh\left(\lnhh \norm{\u}_{L^{\infty} (D_2)} +  h\norm{p}_{L^{\infty} (D_2)}\right) \\
+ C_d\lnhh\left(\twonorm{\u}+ h \norm{\u}_{H^1(\Omega)}  +  h \twonorm{p}\right),
\end{multline}
where for $\xx\in \Omega$, $D_1= B_r(\xx)\cap \Omega$, $D_2= B_{\rr}(\xx)\cap \Omega$, $\rr > r >0$ and $C_d$ depends on $d=\abs{r-\rr}>\bar\kappa h$.

Global pointwise error estimates for the Stokes system similarly to \cref{eq:linfty_global} have been thoroughly discussed in recent years. The three-dimensional $W^{1,\infty}$ case was first discussed in \cite{MR2217368,MR2121575} under smoothness assumptions on the domain or limiting angles in non-smooth domains. Later on, using new results on convex polyhedral domains, e.g. from \cite{MR725151,MR2641539,MR2808700}, the limitations on the domain were weakened in \cite{MR3422453,MR2945141}.
The $L^{\infty}$ bounds were first discussed for $\Omega \subset \R^2$ in \cite{MR935076} and for dimensions greater than one and smooth domains in \cite{MR2217368} but with the $W^{1,\infty}$ norm appearing on the right-hand side and using weighted norms, which is not sufficient for the applications we have in mind.  

Interior (or local) maximum norm estimates are well-known for elliptic equations, see, e.g., \cite{MR3614014,MR0431753}, and are particularly useful when dealing with scenarios where the solution has low regularity close to the boundary or on local subsets of $\Omega$, e.g. for optimal control problems with pointwise state constraints, sparse optimal control and pointwise best approximation results for the time dependent problem, see \cite{Reyes2008,MR3498514,MR3072225}. For the Stokes system, the only pointwise interior error estimates are available on regular translation invariant meshes in two dimensions \cite{MR2297317}.
To our best knowledge, the interior results presented here are novel and have not been discussed before. 

Let us quickly comment on one property specific to the Stokes problem. Regularity results typically appear as velocity-pressure pair where the pressure has lower norm, e.g. $\inftynorm{\nabla \u}$ and $\inftynorm{p}$. This pair can then be estimated as in \cite{MR3422453,MR2945141}. Thus, we only supply estimates for $\inftynorm{\u_h}$ in the max-norm estimate since bounds for $\norm{p_h}_{W^{-1,\infty}(\Omega)}$ would add another layer of complexity and to our knowledge have no apparent advantages.

In three dimensions our proof of the local estimates is essentially based on $L^1$ and weighted estimates of regularized Green's functions. For $W^{1,\infty}$ it is enough to slightly adapt the results from \cite{MR2945141} for the Green's function of velocity and pressure.

In the case of $L^{\infty}$, we prove the respective estimates using the local energy estimates given in \cite{MR2945141} and estimates for Green's matrix of the Stokes system, see, e.g., \cite{MR2641539}.
Furthermore, another important element of the proof for $L^{\infty}$ is a pointwise estimate of the Ritz projection \cite{MR3470741}. Using the stability result proven there, we are able to carry out our proof without the need to discuss the behavior of the discrete solution along finite element boundaries.

In two dimensions our approach for the local estimates follows along the lines of the three-dimensional case. Here the estimates for the regularized Green's functions and the Ritz projection are all known from the literature, see \cite{MR935076,MR2121575,MR551291}. The results from \cite{MR935076,MR2121575} are derived using an alternative technique, the global weighted approach as introduced in \cite{MR0474884,MR0488859}. For the global weighted approach we need similar but slightly different assumptions on the finite element space than for the local energy estimate technique in the three-dimensional setting. Thus, to keep the notation simple, we deal with the two dimensional case in a separate section at the end of this work.

Several important applications from Navier-Stokes free surface flows to the numerical analysis of finite-element schemes for non-Newtonian flows have already been noted in \cite{MR2121575}. As mentioned, interior estimates play a role specifically for optimal control problems with state constraints, e.g. in \cite{MR2472877}.
Stokes optimal control problems are also closely related to subproblems in optimal control of Navier-Stokes systems where for Newton iterations one has to solve linearized optimal control subproblems in each step, see, e.g. \cite{MR2154110}.

An outline of this paper is as follows. In \Cref{section:assumptions_and_main_results}, we introduce notation and state assumptions on the approximation operators as well as the main results of our analysis. \Cref{section:proof_of_main_theorems} gives key arguments for the proof of the main theorems for the velocity and reduces them to the estimates of regularized Green's functions, which are derived in \Cref{section:duality_interpolation}. Based on these results, we deal with bounds for the pressure in \Cref{section:pressure}. Finally, in the last section we show the local estimates in two dimensions.

\section{Assumptions and main results in three dimensions}
\label{section:assumptions_and_main_results}

\subsection{Notation}
We now introduce basic notation. Throughout this paper, we use the usual notation for the Lebesgue, Sobolev and H\"older spaces. These spaces can be extended in a straightforward manner to vector functions, with the same notation but with the following modification for the norm in the non-Hilbert case: if $\u = (u_1,u_2,u_3)$, we then set
\begin{equation}
\norm{\u}_{L^r(\Omega)} = \left[\int_{\Omega} \abs{\u(\x)}^rd\x\right ]^{1/r}
\end{equation}
where $\abs{\:\cdot\:}$ denotes the Euclidean vector norm for vectors or the Frobenius norm for tensors.

We denote by $(\:\cdot\:,\:\cdot\:)$ the $L^2(\Omega)$ inner product and specify subdomains by subscripts in the case they are not equal to the whole domain.
In the analysis, we also make use of the weight $\sigma =\sigma_{\vec{x}_0,h}(\vec x)= \sqrt{\abs{\vec x - \x_0}^2 + (\kappa h)^2}$ for which $\x_0$, $\kappa$ and $h$ will be defined later on.

\subsection{Basic estimates}%
\label{section:continuousproblem}
Next we want to recall some results for solutions to \cref{eq:def_stoke_1,eq:def_stoke_2,eq:def_stoke_3}. Existence and uniqueness of the solutions to the problem on bounded domains are shown in \cite{MR2808162}. For the proof of the respective regularity estimates on convex polyhedral domains we refer to \cite{MR977489,MR2321139}. 
For $\f \in H^{-1}(\Omega)^3$ there holds
\begin{equation}
\norm{\u}_{H^1(\Omega)}+\norm{p}_{L^2(\Omega)}\leq C\norm{\f}_{H^{-1}(\Omega)}. \label{eq:H1_estimate2}
\end{equation}
Furthermore, for $\f \in L^2(\Omega)$, $(\u,p)$ are elements of $(H^1_0(\Omega)\cap H^{2}(\Omega))^3\times H^{1}(\Omega)$ and it holds
\begin{equation}
\norm{\u}_{H^2(\Omega)} + \norm{p}_{H^1(\Omega)} \leq C\twonorm{\f}. \label{eq:H2_estimate}
\end{equation}

\subsubsection{Local \texorpdfstring{{$H^2$}}{H2} stability estimates}
In the following analysis we will also require the following localized $H^2$ stability estimates.

\begin{lemma}
	\label{lemma:localh2}
	Let $A_1 = B_r(\tilde x) \cap \Omega$, $A_2 = B_{\tilde r}(\tilde x) \cap \Omega$ for $\tilde x\in \Omega$ and $\tilde r > r >0$. We denote the difference of the radii by $d=\abs{\tilde r -r}$. Furthermore let $(\u,p)$ be the solution to \cref{eq:def_stoke_1,eq:def_stoke_2,eq:def_stoke_3}. Then, it holds
	\begin{equation}
	\norm{\u}_{H^2(A_1)} + \norm{p}_{H^1(A_1)} \leq C\Big(\norm{\f}_{L^2(A_2)} + \frac{1}{d}\norm{\nabla \u}_{L^2(A_2)}+\frac{1}{d^2}\norm{ \u}_{L^2(A_2)}+\frac{1}{d}\norm{p}_{L^2(A_2)}\Big).
	\end{equation}
\end{lemma}
\begin{proof}
	Let $\omega \in C^{\infty}(\Omega)$ be a smooth cut-off function with $\omega = 1$ on $A_1$ and $\omega=0$ on $\Omega \backslash A_2$ such that
	\begin{equation}
	\abs{\nabla^k \omega}\sim \frac{1}{d^k} \quad \text{for } k=0,1,2. \label{eq:lemma_localh2_1}
	\end{equation}
	We consider $\tilde u = \omega \u$ and $\tilde p = \omega p$. Then, we get the following weak formulation for $\vec\varphi \in H^1_0(\Omega)^3$
	\begin{align}
	(\nabla \tilde u, \nabla \vec\varphi) &= (\nabla \omega \otimes \u + \omega \nabla \u,\nabla\vec\varphi)\\
	&= -(\nabla \cdot (\nabla \omega \otimes \u),\vec\varphi) + (\nabla \u, \nabla(\omega \vec\varphi)) - (\nabla \u, \nabla \omega \otimes \vec\varphi)\\
	&= -(\nabla \cdot (\nabla \omega \otimes \u),\vec\varphi) + (\f, \omega \vec\varphi) + (p,\nabla \cdot (\omega \vec\varphi)) - (\nabla \u, \nabla \omega \otimes \vec\varphi)\\
	&= -(\nabla \cdot (\nabla \omega \otimes \u),\vec\varphi) + (\f, \omega \vec\varphi) + (\omega p,\nabla \cdot \vec\varphi) + (\nabla \omega p,\vec\varphi ) - (\nabla \u \nabla \omega,  \vec\varphi),
	\end{align}
	where we used \cref{eq:def_stoke_1} and in addition we get	$\nabla \cdot \tilde u = \nabla \omega \cdot \u$.
	Thus, $\tilde u$ and $\tilde p$ solve the following boundary value problem in the weak sense
	\begin{subequations}
		\begin{alignat}{3}
		-\Delta \tilde u + \nabla \tilde p &= \f - \nabla \cdot (\nabla \omega \otimes \u) + \nabla \omega p - \nabla \u \nabla \omega
		&\text{ in } A_2, \label{eq:def_stoke_local_1} \\
		\nabla \cdot \tilde u &= \nabla \omega \cdot \u 
		&\text{ in } A_2, \label{eq:def_stoke_local_2}\\
		\tilde u &= \vec 0 
		&\text{ on } \partial A_2. \label{eq:def_stoke_local_3}
		\end{alignat}
	\end{subequations}
	By construction we have that $A_2$ is convex and $\nabla \omega \cdot \u$ vanishing on the boundary $\partial A_2$. Thus, according to  \cite[~Thm.~9.20]{MR977489} and the fact that $\nabla \cdot \tilde u$ is zero on $\partial A_2$, the $H^2$ regularity result  \cref{eq:H2_estimate} holds in this situation as well, and we obtain
	\begin{align}
	\norm{\tilde u}_{H^2(A_2)} + &\norm{\tilde p}_{H^1(A_2)} \\&\leq C\Big(\norm{\f}_{L^2(A_2)} + \norm{\nabla \omega \nabla \u}_{L^2(A_2)} +\norm{\nabla^2 \omega  \u}_{L^2(A_2)}+\norm{\nabla \omega p}_{L^2(A_2)}\Big)\\
	&\leq C\Big(\norm{\f}_{L^2(A_2)} + \frac{1}{d}\norm{\nabla \u}_{L^2(A_2)} +\frac{1}{d^2}\norm{\u}_{L^2(A_2)}+\frac{1}{d}\norm{p}_{L^2(A_2)}\Big),
	\end{align}
	where we used \cref{eq:lemma_localh2_1}. We get
	\begin{multline}
	\norm{\u}_{H^2(A_1)} + \norm{p}_{H^1(A_1)} = \norm{\tilde u}_{H^2(A_1)} + \norm{\tilde p}_{H^1(A_1)}
	\leq \norm{\tilde u}_{H^2(A_2)} + \norm{\tilde p}_{H^1(A_2)}\\
	\leq C\left(\norm{\f}_{L^2(A_2)} + \frac{1}{d}\norm{\nabla \u}_{L^2(A_2)} +\frac{1}{d^2}\norm{\u}_{L^2(A_2)}+\frac{1}{d}\norm{p}_{L^2(A_2)}\right).
	\end{multline}
\end{proof}
Using a covering argument (see \Cref{corollary:w1_localization} for details),  we may show the following corollary.
\begin{corollary}
	\label{corollary:localh2}
	Let $\Omega_1 \subset \Omega_2 \subset \Omega$ with $dist(\bar \Omega_1, \partial \Omega_2)\geq d$, then holds for $(\u,p)$ the solution to \cref{eq:def_stoke_1,eq:def_stoke_2,eq:def_stoke_3} that
	\begin{equation}
	\norm{\u}_{H^2(\Omega_1)} + \norm{p}_{H^1(\Omega_1)} \leq C\Big(\norm{\f}_{L^2(\Omega_2)} + \frac{1}{d}\norm{\nabla \u}_{L^2(\Omega_2)}+\frac{1}{d^2}\norm{ \u}_{L^2(\Omega_2)}+\frac{1}{d}\norm{p}_{L^2(\Omega_2)}\Big).
	\end{equation}
\end{corollary}

\subsubsection{Green's matrix estimate}
We also need estimates of the respective Green's matrix for the Stokes problem. For this, refer to \cite[Section 11.5]{MR2641539}. Let $\phi\in C^{\infty}(\bar\Omega)$ be vanishing in a neighborhood of the edges and
$\int_{\Omega} \phi (\x)d\x =1.$
The matrix $G(\x,\y) = (G_{i,j}(\x,\y))_{i,j=1,2,3,4}$ is the Green's matrix for problem \cref{eq:def_stoke_1,eq:def_stoke_2,eq:def_stoke_3} if the vector functions $\vec G_j = (G_{1,j},G_{2,j},G_{3,j})^T$ and $G_{4,j}$ are solutions of the problem
\begin{alignat}{3}
-\Delta_x \vec G_j(\x,\y) + \nabla_x G_{4,j}(\x,\y) &= \delta (\x-\y) (\delta_{1,j},\delta_{2,j},\delta_{3,j})^t &&\quad \text{for } \x, \y \in \Omega \\
-\nabla_x \cdot \vec G_j(\x,\y) &= (\delta (\x-\y) -\phi(\x))\delta_{4,j} &&\quad \text{for } \x, \y \in \Omega,\\
\vec G_j(\x,\y)&= \vec 0 &&\quad \text{for } \x \in \partial \Omega, \y \in \Omega
\end{alignat}
and  $G_{4,j}$ satisfies the condition
\begin{equation}
\int_{\Omega} \G_{4,j}(\x,\y) \phi(\x) d\x =0 \quad \text{for } \y \text{ in } \Omega, j = 1,2,3,4.
\end{equation}
For the existence and uniqueness of such a matrix, we again refer to \cite{MR2641539}.
If now $f\in H^{-1}(\Omega)^3$ and the uniquely determined solutions of the Stokes system given by $(\u,p)\in  H^{1}_0(\Omega)^3 \times L_2(\Omega)$ satisfy the condition
\begin{equation}
\int_{\Omega}p(\x)\phi(\x) d\x=0 \label{eq:cond_pressure}
\end{equation}
then the components of $(\u,p)$ admit the representations
\begin{equation}
\u_i(\x) = \int_{\Omega}\vec f(\vec{\xi})\cdot \G_i(\vec{\xi},\x)d\vec\xi , \quad i =1,2,3,\quad
p(\x) = \int_{\Omega} \vec f(\vec{\xi}) \cdot \G_4 (\vec{\xi},\x)d\vec\xi . \label{eq:greens_matrix_2}
\end{equation}
To apply this result to our case, we need to find a suitable $\bar \phi$ such that \cref{eq:cond_pressure} holds. We show this is possible for $p \in C^{0,\alpha}(\Omega)\cap L^2_0(\Omega)$. By \cite[Theorem 11.3.2]{MR2641539} this is fulfilled for data in $C^{-1,\alpha}(\Omega)$. For our use cases in later sections we consider at least continuous right-hand sides, so this is applicable.

Without loss of generality, we assume $p \neq 0$. Thus, since the mean value of $p$ is zero, there exist non-empty open sets $A,B \Subset \Omega$ such that
$p > 0$ on $A$ and
$p < 0$ on $B$.
We then can choose $\bar \phi$ such that
$\bar \phi = 0$ on $\Omega \backslash (A \cup B)$ and
$\bar \phi >0$ on $A$, $B$
and thus $\bar \phi$ vanishing close to the edges of $\Omega$.
Through suitable scaling on $A$ and $B$, we get
\begin{equation}
\int_{A}p(\x) \bar \phi(\x) d\x = - \int_{B}p(\x) \bar \phi(\x) d\x
\end{equation}
and hence we can conclude that \cref{eq:cond_pressure} holds for $\bar \phi(\x)$.
Finally, since by assumption $\bar\phi >0$, we normalize $\bar\phi$ with respect to the $L^1(\Omega)$ norm to complete the construction. This shows that we can apply the results for the Green's matrix to $(\u,p)$. Furthermore, we can also use the available results from \cite{MR2945141}.

We state estimates for the Green's matrix specific to convex polyhedral domains as it can be found in \cite[Theorem 11.5.5, Corollary 11.5.6]{MR2641539}.
\begin{proposition}
	\label{proposition:Green_matrix}
	Let $\Omega$ be a convex polyhedral type domain. Then, the elements of the matrix $G(\x,\vec\xi)$ satisfy the estimate
	\begin{equation}
	\abs{\partial^{\theta}_x\partial^{\beta}_{\xi} G_{i,j}(\x,\vec\xi)} \leq c \abs{\x-\vec\xi}^{-1-\delta_{i,4} - \delta_{j,4}-\abs{\theta}-\abs{\beta}}
	\end{equation}
	for $\abs{\theta}\leq 1- \delta_{i,4}$ and $\abs{\beta}\leq 1- \delta_{j,4}$.
	Furthermore, the following H\"older type estimate holds in this setting
	\begin{equation}
	\frac{\abs{\partial_{\xi}^{\theta}G_{i,j}(\x,\vec\xi)-\partial_{\xi}^{\theta}G_{i,j}(\y,\vec\xi)}}{\abs{\x-\y}^{\alpha}}
	\leq C\Big(\abs{\x-\vec\xi}^{-1-\alpha -\delta_{j,4}-\abs{\theta}} + \abs{\y-\vec\xi}^{-1-\alpha -\delta_{j,4}-\abs{\theta}}\Big).
	\end{equation}
\end{proposition}

\subsection{Finite element approximation}
\label{ss:fem_approximation}
Let $\Th$ be a regular, quasi-uniform family of triangulations of $\bar \Omega$, made of closed tetrahedra $T$, where $h$ is the global mesh-size and $L^2_0(\Omega)$ the space of $L^2(\Omega)$ functions with zero-mean value. Let $\V_h \subset H^1_0(\Omega)^3$ and $M_h \subset L^2_0(\Omega)$ be a pair of finite element spaces satisfying a uniform discrete inf-sup condition,
\begin{equation}
\sup_{\v_h \in \V_h}\frac{(q_h,\nabla \cdot \v_h) }{\twonorm{\nabla \v_h}} \geq \beta \twonorm{q_h} \quad \forall q_h \in M_h, \label{eq:discrete_infsup}
\end{equation}
with a constant $\tilde\beta >0$ independent of $h$. The respective discrete solution associated with the velocity-pressure pair $(\u,p)\in H^1_0(\Omega)^3 \times L^2_0(\Omega)$ is defined as the pair $(\u_h,p_h) \in \V_h \times M_h$ that solves the weak form of \cref{eq:def_stoke_1,eq:def_stoke_2,eq:def_stoke_3} given by the bilinear form $a(\cdot,\cdot)$ which is defined as
\begin{equation}
a((\u_h,p_h),(\v_h,q_h)) = (\nabla \u_h, \nabla \v_h) - (p_h, \nabla \cdot \v_h) + (\nabla \cdot \u_h, q_h) \label{eq:def_stoke_fem}
\end{equation}
and the equation
\begin{equation}
a((\u_h,p_h),(\v_h,q_h)) = (\vec{f}, \v_h) \quad \forall (\v_h,q_h) \in \V_h \times M_h. \label{eq:def_stoke_fem_equation}
\end{equation}

\subsection{Assumptions}
\label{ss:interpolation_assumptions}
Next, we make assumptions on the finite element spaces. We assume, there exist approximation operators $P_h$ and $r_h$ as in \cite{MR2945141}, i.e. $P_h$ and $r_h$ fulfill the following properties.  
Let $Q\subset Q_d \subset \Omega$, with $d \geq \bar{\kappa}h$, for some fixed $\bar\kappa$ sufficiently large and $Q_d = \{\x \in \Omega : dist(\x,Q)\leq d \}$.
For $P_h \in \mathcal{L}(H^1_0(\Omega)^3;V_h)$ and $r_h \in \mathcal{L}(L^2(\Omega);\bar M_h)$ with $\bar M_h$ corresponding to $M_h$ without the zero-mean value constraint, we assume the following assumptions hold.
\begin{assumption}[Stability of $P_h$ in $H^1(\Omega)^3$]
	\label{assumption:stability}
	There exists a constant $C$ independent of $h$ such that
	\begin{equation}
	\twonorm{\nabla P_h (\v)} \leq C \twonorm{\nabla \v}, \quad \forall \v \in H^1_0(\Omega)^3.
	\end{equation}
\end{assumption}

\begin{assumption}[Preservation of discrete divergence for $P_h$]
	\label{assumption:divergence}
	It holds
	\begin{equation}
	(\nabla \cdot (\v - P_h(\v)),q_h) = 0, \quad \forall q_h \in \bar M_h, \quad \forall \v \in H^1_0(\Omega)^3.
	\end{equation}
\end{assumption}

\begin{assumption}[Inverse Inequality]
	\label{assumption:inverse_inequality}
	There is a constant $C$ independent of $h$ such that
	\begin{equation}
	\norm{\v_h}_{W^{1,p}(Q)} \leq Ch^{-1}\norm{\v_h}_{L^{p}(Q_d)} \quad \forall \v_h\in\V_h, 1\leq p \leq \infty.
	\end{equation}
\end{assumption}

\begin{assumption}[$L^2$ approximation]
	\label{assumption:approximation}
	For any $\v \in H^2(\Omega)^3$ and any $q\in H^1(\Omega)$ exists $C$ independent of $h$, $\v$ and $q$ such that
	\begin{align}
	\norm{P_h(\v) - \v}_{L^2(Q)} + h \norm{\nabla(P_h(\v)-\v)}_{L^2(Q)} &\leq C h^2\norm{\nabla^2 \v}_{L^2(Q_d)},\\
	\norm{r_h(q) - q}_{L^2(Q)} &\leq C h \norm{\nabla q}_{L^2(Q_d)}.
	\end{align}
\end{assumption}
In the following, $\vec e_i$ denotes the $i$-th standard basis vector in $\R^3$. 
\begin{assumption}[Approximation in the H\"older spaces]
	\label{assumption:approximation_hoelder}
	\\ For $\v \in \big(C^{1,\alpha}(\Omega)\cap H^1_0(\Omega)\big)^3$ and $q \in C^{0,\alpha}(\Omega)$, it holds
	\begin{align}
	\norm{\nabla (P_h(\v)-\v)}_{L^{\infty}(Q)} &\leq Ch^{\alpha} \norm{\v}_{C^{1,\alpha}(Q_d)},\\
	\norm{r_h(q) - q}_{L^{\infty}(Q)} &\leq Ch^{\alpha}\norm{q}_{C^{0,\alpha}(Q_d)},
	\end{align}
	where
	\begin{equation}
	\norm{\v}_{C^{1+\alpha}(Q)}= \norm{\v}_{C^1(Q)}+ \sup_{\substack{\x,\y \in Q\\ i\in \{1,2,3\}}}\frac{\abs{\vec e_i\cdot\nabla(\v(\x)-\v(\y))}}{\abs{\x-\y}^{\alpha}}.
	\end{equation}
\end{assumption}

\begin{assumption}[Super-Approximation I]
	\label{assumption:superapproxmation_dyadic}
	Let $\v_h \in \V_h$ and $\omega \in C^{\infty}_0(Q_d)$ a smooth cut-off function such that $\omega \equiv 1$ on $Q$ and
	\begin{equation}
	\abs{\nabla^s\omega} \leq C d^{-s}, \quad s=0,1,
	\end{equation}
	where $Q_d=\{\vec{x}\in \Omega:\ dist(\vec{x},\partial Q)\geq d\}$.
	We assume 
	\begin{equation}
	\norm{\nabla(\omega^2 \v_h - P_h(\omega^2\v_h))}_{L^2(Q)}\leq C d^{-1}\norm{\v_h}_{L^2(Q_d)}.
	\end{equation}
	For $q_h \in \bar M_h$, we assume
	\begin{equation}
	\norm{\omega^2 q_h -  r_h(\omega^2q_h)}_{L^2(Q)}\leq C hd^{-1}\norm{q_h}_{L^2(Q_d)}.
	\end{equation}
\end{assumption}

One common example of a finite element space satisfying the above assumptions are the Taylor-Hood finite elements of order greater or equal than three. For more details on these spaces and the respective approximation operators, we refer to \cite{MR1342712,MR2121575,MR3422453}.

\begin{remark}
	Here we restrict ourselves to Taylor-Hood finite element spaces since in the following arguments we use results for finite element approximations of elliptic problems. These results are available for the usual space of Lagrange finite elements and can possibly be extended to other elements used for the Stokes problem, like e.g. the ``mini" element, which also fulfills the assumptions above.
\end{remark}

Next, we state a well-known energy error estimate for an approximation of the Stokes system. For details on the proof, see e.g. \cite[Proposition 4.14]{MR2050138}.
\begin{proposition}
	\label{proposition:fem_convergence}
	Let $(\u,p)$ solve \cref{eq:def_stoke_1,eq:def_stoke_2,eq:def_stoke_3} and $(\u_h,p_h)$ be its finite element approximation defined by \cref{eq:def_stoke_fem_equation}. Under the assumptions above, there exists a constant $C$ independent of $h$ such that,
	\begin{equation}
	\norm{\u-\u_h}_{H^1(\Omega)} + \norm{p-p_h}_{L^2(\Omega)} 
	\leq C \min_{(\v_h,q_h) \in \V_h\times M_h}\left(\norm{\u-\v_h}_{H^1(\Omega)} + \norm{p-q_h}_{L^2(\Omega)}\right).
	\end{equation}
\end{proposition}

\subsection{Local energy estimates}
An important tool in our analysis are the
local energy estimates from \cite[Thm.~2]{MR2945141}.
\begin{proposition}\label{proposition:local_energy_estimate}
	Suppose $(\v,q)\in H^1_0(\Omega)^3\times L^2(\Omega)$ and $(\v_h,q_h) \in \V_h \times M_h$ satisfy 
	\begin{equation}
	a((\v-\v_h,q-q_h),(\vchi,w)) = 0 \quad \forall (\vchi,w) \in \V_h \times M_h
	\end{equation}
	for the bilinear form $a(\cdot,\cdot)$ given in \cref{eq:def_stoke_fem}. Then, there exists a constant $C$ such that for every pair of sets $A_1 \subset A_2 \subset \Omega$ such that $dist(\bar A_1, \partial A_2 \backslash \partial \Omega) \geq d \geq \bar\kappa h$ (for some fixed constant $\bar\kappa$ sufficiently large) the following bound holds for every $\varepsilon>0$
	\begin{multline}
	\norm{\nabla(\v-\v_h)}_{L^2(A_1)} \leq C\norm{\nabla (\v - P_h (\v))}_{L^2(A_2)}+ C\norm{q- r_h (q)}_{L^2(A_2)} \\\nonumber
	+\frac{C}{\varepsilon d}\norm{\v- P_h(\v)}_{L^2(A_2)}
	+ \varepsilon \norm{\nabla(\v-\v_h)}_{L^2(A_2)} + \frac{C}{\varepsilon d}\norm{\v -\v_h}_{L^2(A_2)}.
	\end{multline}
\end{proposition}

\subsection{Main results}
In the following statements, the constant $C$ is independent of $\u$, $p$ and $h$, but may depend on the parameter $\alpha$ related to the largest interior angle of $\partial \Omega$.
We start with the $W^{1,\infty}$ error estimates. 
The global stability result 
\begin{equation}\label{eq: Global W1infty}
\norm{\nabla \u_h}_{L^{\infty}(\Om)} + \norm{p_h}_{L^{\infty}(\Om)} \le C\left(\norm{\nabla\u}_{L^{\infty}(\Om)} + \norm{p}_{L^{\infty}(\Om)} \right),
\end{equation}
on convex polyhedral domains was established in \cite{MR2945141} (see also \cite{MR3422453}). Here, we establish a localized version of it.
In the our results $B_r(\xx)$ denotes  a ball of radius $r$ centered at $\xx \in \Omega$.
\begin{theorem}[Interior $W^{1,\infty}$ estimate for the velocity and $L^{\infty}$ estimate for the pressure]
	\label{theorem:W1infty_localization}
	Let the assumptions of \Cref{ss:fem_approximation} and \Cref{ss:interpolation_assumptions} hold. Put $D_1= B_r(\xx)\cap \Omega$, $D_2= B_{\rr}(\xx)\cap \Omega$, $\rr>r>\bar \kappa h$ (with $\bar \kappa$ large enough), $d=\rr-r\geq \bar\kappa h$. If $(\u,p)\in (W^{1,\infty}(D_2)^3 \times L^{\infty}(D_2))\cap (H^1_0(\Omega)^3 \times L^2_0(\Omega))$ is the solution to \cref{eq:def_stoke_1,eq:def_stoke_2,eq:def_stoke_3}, and $(\u_h,p_h)$ is the solution to \cref{eq:def_stoke_fem_equation}, then 
	\begin{multline}
	\norm{\nabla \u_h}_{L^{\infty}(D_1)} + \norm{p_h}_{L^{\infty}(D_1)}
	\\ \leq C\left(\norm{\nabla\u}_{L^{\infty}(D_2)} + \norm{p}_{L^{\infty}(D_2)} \right)
	+C_d\Big(\twonorm{\nabla\u} + \twonorm{p}\Big). \nonumber
	\end{multline}
	Here, the constant $C_d$ depends on the distance of $B_{r}(\xx)$ from $\partial B_{\rr}(\xx)$.
\end{theorem}
Next we state similar results for the velocity in $L^\infty$ norm. 
\begin{theorem}[Global $L^{\infty}$ estimate for the  velocity]
	\label{theorem:Linfty_nonzerodivergence}
	Under the assumptions of \Cref{ss:fem_approximation} and \Cref{ss:interpolation_assumptions}, for $(\u,p)\in (L^{\infty}(\Omega)^3 \times L^{\infty}(\Omega))\cap (H^1_0(\Omega)^3 \times L^2_0(\Omega))$  the solution to \cref{eq:def_stoke_1,eq:def_stoke_2,eq:def_stoke_3} and $(\u_h,p_h)$ the solution to \cref{eq:def_stoke_fem_equation}, it holds
	\begin{equation}
	\inftynorm{\u_h} \leq C\lnhh\Big(\lnhh\inftynorm{\u}+ h\inftynorm{p}\Big).
	\end{equation}
\end{theorem}
The additional logarithmic factor in front of the velocity is probably not optimal, it appears when applying a pointwise estimate for the Ritz projection. We also get the respective local estimates.
\begin{theorem}[Interior $L^{\infty}$ error estimate for the velocity]
	\label{theorem:Linfty_localization}
	Under the assumptions of \Cref{ss:fem_approximation} and \Cref{ss:interpolation_assumptions}, with $D_1= B_r(\xx)\cap \Omega$, $D_2= B_{\rr}(\xx)\cap \Omega$, $\rr>r>\bar \kappa h$ (with $\bar \kappa$ large enough), $d=\rr-r\geq \bar\kappa h$ and for $(\u,p)\in (L^{\infty}(D_2)^3 \times L^{\infty}(D_2))\cap (H^1_0(\Omega)^3 \times L^2_0(\Omega))$  the solution to \cref{eq:def_stoke_1,eq:def_stoke_2,eq:def_stoke_3} and $(\u_h,p_h)$ the solution to \cref{eq:def_stoke_fem_equation}, it holds
	\begin{multline}
	\norm{\u_h}_{L^{\infty}(D_1)}\nonumber
	\leq C\lnhh\left(\lnhh \norm{\u}_{L^{\infty} (D_2)} +  h\norm{p}_{L^{\infty} (D_2)}\right) \\
	+ C_d\lnhh\left( h \norm{\u}_{H^1(\Omega)} +  \twonorm{\u} +  h \twonorm{p}\right).
	\end{multline}
	Here, the constant $C_d$ depends on the distance of $B_{r}(\xx)$ from $\partial B_{\rr}(\xx)$.
\end{theorem}
Based on these theorems, we can derive the following corollaries for general subdomains $\Omega_1 \subset \Omega_2 \subset \Omega$ with $dist(\bar\Omega_1, \partial \Omega_2)\geq d\geq \bar \kappa h$.
\begin{corollary}[Interior $W^{1,\infty}$ estimate for the velocity and $L^{\infty}$ estimate for the pressure]
	\label{corollary:w1_localization}
	Under the assumptions of \Cref{ss:fem_approximation} and \Cref{ss:interpolation_assumptions}, $\Omega_1 \subset \Omega_2 \subset \Omega$ with $dist(\bar \Omega_1, \partial \Omega_2)\geq d \geq \bar \kappa h$ and for $(\u,p)\in (W^{1,\infty}(\Omega_2)^3 \times L^{\infty}(\Omega_2))\cap (H^1_0(\Omega)^3 \times L^2_0(\Omega))$ the solution to \cref{eq:def_stoke_1,eq:def_stoke_2,eq:def_stoke_3} and $(\u_h,p_h)$ the solution to \cref{eq:def_stoke_fem_equation}, we have
	\begin{multline}
	\norm{\nabla \u_h}_{L^{\infty}(\Omega_1)} +\norm{p_h}_{L^{\infty}(\Omega_1)}\nonumber
	\leq C\left(\norm{\nabla\u}_{L^{\infty}(\Omega_2)} + \norm{p}_{L^{\infty}(\Omega_2)} \right) \\
	+C_d\Big(\twonorm{\nabla\u} +\twonorm{p}\Big).
	\end{multline}
	Here, the constant $C_d$ depends on the distance to $\Omega_1$ from $\partial \Omega_2$.
\end{corollary}
\begin{proof}
	We can construct a covering $\{K_i\}^M_{i=1}$ of $\Omega_1$, with $K_i = B_{\rr_i}(\xx_i)\cap \Omega_1$ such that
	\begin{enumerate}[label={(\arabic*)}]
		\item $\Omega_1 \subset \bigcup^M_{i=1}K_i$.
		\item $\xx_i \in \bar \Omega_1$ for $1\leq i \leq M$.
		\item Let $L_i = B_{r_i}(\xx_i)\cap \Omega_2$ where $r_i = \rr_i+d$. There exists a fixed number $N$ such that each point $\x \in \bigcup^M_{i=1}L_i$ is contained in at most $N$ sets from $\{L_j\}^M_{j=1}$.
	\end{enumerate}
	Now, since $dist(\bar\Omega_1, \partial \Omega_2)\geq d$ and (2), we have that $\bigcup_{i=1}^M \subset \Omega_2$. We can apply \Cref{theorem:W1infty_localization} to the pairs $K_i \subset L_i$:
	\begin{align}
	\norm{\nabla \u_h}_{L^{\infty}(\Omega_1)} &+\norm{p_h}_{L^{\infty}(\Omega_1)} \leq \sum_{i=1}^{M} \norm{\nabla \u_h}_{L^{\infty}(K_i)} +\norm{p_h}_{L^{\infty}(K_i)} \\
	&\leq \sum_{i=1}^{M} \Big(C\left(\norm{\nabla\u}_{L^{\infty}(L_i)} + \norm{p}_{L^{\infty}(L_i)} \right) +C_d\left(\twonorm{\nabla\u} + \twonorm{p}\right)\Big)\\
	&\leq N\Big(C\left(\norm{\nabla\u}_{L^{\infty}(\Omega_2)} + \norm{p}_{L^{\infty}(\Omega_2)} \right)
	+C_d\left(\twonorm{\nabla\u} + \twonorm{p}\right)\Big),
	\end{align}
	where we used (3) in the third line.
\end{proof}
Similarly, the following corollary follows with $dist(\bar \Omega_1, \partial \Omega_2)\geq d$.
\begin{corollary}[Interior $L^{\infty}$ error estimate for the velocity]
	Under the assumptions of \Cref{ss:fem_approximation} and \Cref{ss:interpolation_assumptions}, $\Omega_1 \subset \Omega_2 \subset \Omega$ with $dist(\bar\Omega_1, \partial \Omega_2)\geq d \geq \bar \kappa h$ and for $(\u,p)\in (L^{\infty}(\Omega_2)^3 \times L^{\infty}(\Omega_2))\cap (H^1_0(\Omega)^3 \times L^2_0(\Omega))$ the solution to \cref{eq:def_stoke_1,eq:def_stoke_2,eq:def_stoke_3} and  $(\u_h,p_h)$ the solution to \cref{eq:def_stoke_fem_equation}, we have
	\begin{multline}
	\norm{\u_h}_{L^{\infty}(\Omega_1)}\nonumber
	\leq C\lnhh\Big(\lnhh \norm{\u}_{L^{\infty} (\Omega_2)} + h\norm{p}_{L^{\infty} (\Omega_2)}\Big) \\
	+ C_d\lnhh\Big( h \norm{\u}_{H^1(\Omega)}+ \twonorm{u} + h \twonorm{p} \Big).
	\end{multline}
	Here, the constant $C_d$ depends on the distance to $\Omega_1$ from $\partial \Omega_2$.
\end{corollary}
\begin{remark}
	We may also write the results above in terms of best approximation estimates. For example for $L^{\infty}$ global bounds:
	\begin{equation}
	\inftynorm{\u-\u_h} 
	\leq \inf_{(\v_h,q_h)\in \V_h \times M_h} C\lnhh \Big(\lnhh\inftynorm{\u - \v_h} + h\inftynorm{p - q_h}\Big).
	\end{equation}
	Naturally, this also applies for other results in this section.
\end{remark}
\begin{remark}
	Using the weighted discrete $\inf$-$\sup$ condition from \cite{MR1880723} it is possible to extend the the global estimate to the compressible case. However, for the applications we have in mind the incompressible Stokes system is sufficient.
\end{remark}

\section{Proof of main theorems}
\label{section:proof_of_main_theorems}
In this section, we reduce the proofs of
\Cref{theorem:W1infty_localization,theorem:Linfty_nonzerodivergence,theorem:Linfty_localization} for the velocity to certain estimates for the regularized Green's functions. The estimates for the pressure are given in \Cref{section:pressure}. 
To introduce the regularized Green's function we first need to introduce a regularized delta function. In addition we will require a certain weight function.
\subsection{Regularized delta function and the weight function}\label{section:technical_tools_and_estimates}

Let $R>0$ such that for any $\vec x \in \Omega$ the ball $B_{R}(\vec x)$ contains $\Omega$. Furthermore, let $\x_0$ be an arbitrary point of $\bar{\Omega}$ and $T_{\x_0}\in \Th$. In the following sections, we estimate $\abs{\partial_{x_j}\u_{h,i}(\x_0)}$, $\abs{\u_{h,i}(\x_0)}$ for arbitrary $1\leq i, j, \leq 3$ and $\abs{p(\x_0)}$. 

Next we introduce the parameters for the weight function $\sigma(\x)$. Parameter $\kappa > 1$ is a constant that is chosen to be large enough. Furthermore, let $h$ be suitably small such that $\kappa h \leq R$ (see also \cite[Remark 1.4]{MR2121575}).
In the following, we use a regularized Green's function to express the $L^{\infty}(\Omega)$ norm such that the problem is reduced to estimating the discretization error of the Green's function in the $L^1(\Omega)$ norm as in \cite{MR3422453,MR2945141}.
To that end, we define a smooth delta function $\delta_h\in C^1_0(T_{\x_0})$, which satisfies for every $\v_h\in \V_h$:
\begin{align}
\v_{h,i}(\x_0) &= (\v_h, \delta_h \vec e_i)_{T_{\x_0}} \label{eq:def_delta}\\
\norm{\delta_h}_{W^k_q(T_{\x_0})}&\leq C h^{-k-3(1-1/q)}, \quad 1 \leq q \leq \infty,\quad k=0, 1, \dots \label{eq:deltah_est} 
\end{align}
The construction of such a $\delta_h$ can be found in \cite[Appendix]{MR1297478}.
We recall some properties for $\sigma$ and $\delta_h$. 
By construction, it follows
\begin{equation}
\inf_{\x\in \Omega}\sigma(\x)\geq \kappa h. \label{eq:lower_bound_sigma}
\end{equation}
Next, we provide an estimate for the $L^2(\Omega)$ norm of the product of $\delta_h$ and $\sigma$.

\begin{lemma}
	\label{lemma:sigma_delta}
	There exists a constant $C$ such that for $\nu>0$
	\begin{equation}
	\twonorm{\sigma^{\nu}\nabla^k \delta_h} \leq 2^{\nu/2}C\kappa^{\nu}h^{\nu-k-3/2} \quad k=0,1.
	\end{equation}
\end{lemma}
\begin{proof}
	This follows from the fact that $\delta_h$ is only non-zero on $T_{\x_0}$, $\sigma$ is bounded on $T_{\x_0}$ by $\sqrt{2}\kappa h$ and \cref{eq:deltah_est}.
\end{proof}

The general strategy for proving the local results is to partition the domain into the local part and its complement. Then, we use regularized Green's function estimates in the $L^1$ norm on the local part and weighted $L^2$ norm on the complement. For the $L^{\infty}$ error estimates we additionally require a certain estimate for the Ritz projection.

\subsection{Estimates for \texorpdfstring{$W^{1,\infty}(\Omega)$}{W(1,infty)(Omega)}}

The proof of local  $W^{1,\infty}(\Omega)$ error estimates is similar to the global case \cite{MR3422453,MR2945141} and is obtained by introducing a regularized Green's function.
\subsubsection{Regularized Green's function}
For the $W^{1,\infty}$ error estimates, we define the regularized Green's function $(\g_1,\lambda_1)\in H^1_0(\Omega)^3\times L^2_0(\Omega)$  as the solution  to
\begin{subequations}
	\begin{alignat}{3}
	-\Delta \g_1 + \nabla \lambda_1 &= (\partial_{x_j}\delta_h) \vec e_i &&\quad \text{ in } \Omega, \label{eq:def_stoke_green1_1}\\
	\nabla \cdot \g_1 &= 0 &&\quad \text{ in } \Omega,\label{eq:def_stoke_green1_2} \\
	\g_1 &= \vec 0 &&\quad \text{ on } \partial \Omega.\label{eq:def_stoke_green1_3}
	\end{alignat}
\end{subequations}
We also define the finite element approximation $(\g_{1,h},\lambda_{1,h}) \in \V_h \times M_h$ by
\begin{equation}
a((\g_1 - \g_{1,h},\lambda_1-\lambda_{1,h}),(\v_h,q_h)) = 0 \quad \forall (\v_h,q_h) \in \V_h \times M_h.\label{eq:def_stoke_green1_fem}
\end{equation}
\subsubsection{Auxiliary results for \texorpdfstring{{$(\g_1,\lambda_1)$ and $(\g_{1,h},\lambda_{1,h})$}}{(g1,lambda1) and (g1h,lambda1h)}}

To show our main interior $W^{1,\infty}$ result, we need the regularized Green's function error estimate in $L^1(\Omega)$ norm which is given in \cite[Lemma 5.2]{MR2945141}.
\begin{lemma}
	\label{proposition:g1_l1}
	There exists a constant $C$ independent of $h$ and $\g_1$ such that
	\begin{equation}
	\onenorm{\nabla (\g_1-\g_{1,h})} \leq C.
	\end{equation}
\end{lemma}
In addition, we also need the following  weighted  estimate, the proof of which follows by a minor modification of the proof in \cite[Lemma 5.2]{MR2945141}.
\begin{corollary}
	\label{corollary:g1_sigma}
	There exists a constant $C$ independent of $h$ and $\g_1$ such that
	\begin{equation}
	\twonorm{\sigma^{3/2}\nabla (\g_1-\g_{1,h})} \leq C.
	\end{equation}
\end{corollary}
The details on the proof of this corollary are given in \Cref{section:duality_interpolation} where we introduce the respective dyadic decomposition.

\begin{remark}
	The results in \Cref{proposition:g1_l1} and \Cref{corollary:g1_sigma} also follow in a straightforward manner from the arguments in \cite{MR3422453} but are not available in our setting since we make different assumptions on the finite element space which we find similar but not directly compatible to the assumptions made in \cite{MR3422453}.
\end{remark}
\subsubsection{Localization}
\label{section:W1infty_localization}
We reduce the proof to estimates involving $\g_1$ and $\g_{1,h}$.
\begin{proof}[Proof of \Cref{theorem:W1infty_localization} (velocity)]
	Using the regularized Green's function as defined in \cref{eq:def_stoke_green1_1,eq:def_stoke_green1_2,eq:def_stoke_green1_3},  for $\x_0 \in T_{\x_0}\subset D_1$, we have as in \cite{MR2945141}
	\begin{align}
	&-\partial_{x_j}(\u_h)_i(\x_0) = (\u_h,(\partial_{x_j}\delta_h)\vec{e}_i) \tag{by \cref{eq:def_delta}}\\
	&= (\u_h, -\Delta \g_1 + \nabla \lambda_1) \tag{by \cref{eq:def_stoke_green1_1}} \\
	&= (\nabla \u_h, \nabla \g_1) + (\u_h, \nabla \lambda_1)\\
	&= (\nabla \u_h, \nabla \g_1) + (\u_h, \nabla \lambda_{1,h}) + (\nabla \u_h, \nabla (\g_{1,h}-\g_1)) \tag{by \cref{eq:def_stoke_green1_fem}} \\
	&= (\nabla \u_h, \nabla \g_{1,h}) \tag{discrete divergence} \\
	&= (\nabla \u, \nabla \g_{1,h}) + (p-p_h,\nabla \cdot \g_{1,h}) \tag{by \cref{eq:def_stoke_1,eq:def_stoke_fem_equation}} \\
	&= (\nabla \u, \nabla \g_{1,h}) + (p,\nabla \cdot \g_{1,h})\tag{by \cref{eq:def_stoke_green1_fem} and \cref{eq:def_stoke_green1_2}}\\
	&= (\nabla \u, \nabla (\g_{1,h}-\g_{1})) + (\nabla \u, \nabla \g_{1}) + (p,\nabla \cdot (\g_{1,h}-\g_1))\tag{continuous divergence}\\
	&:= I_1 + I_2 +I_3.
	\end{align}
	To treat $I_2$ we use integration by parts, the H\"{o}lder estimate, and \eqref{eq:deltah_est}
	\begin{equation}
	I_2 = (\u,-\Delta \g_1)+(\u,\nabla \lambda_1)=(\u, (\partial_{x_j}\delta_h)\vec{e_i})= (-\partial_{x_j}\u, \delta_h\vec{e_i}) \leq C\norm{\nabla\u}_{L^{\infty}(T_{\x_0})}.
	\end{equation}
	Since $r-\rr>\bar\kappa h$ this proves the result for $I_2$.
	
	For the other two terms,	we split the domain into $D_2$ and $\Omega \backslash D_2$. Using that $\sigma^{-1}>(\bar\kappa (\rr-r))^{-1}$ on $\Omega \backslash D_2$ and the H\"{o}lder estimates, we have
	\begin{align}
	I_1 +I_3 &\leq C\left(\norm{\nabla \u}_{L^{\infty}(D_2)} + \norm{p}_{L^{\infty}(D_2)}\right)\onenorm{\nabla(\g_{1,h}-\g_1)}\\
	&\quad + C\Big(\norm{\sigma^{-3/2}\nabla \u}_{L^2(\Omega\backslash D_2)} + \norm{\sigma^{-3/2}p}_{L^2(\Omega\backslash D_2)}\Big)\twonorm{\sigma^{3/2}\nabla(\g_{1,h}-\g_1)}\\
	&\leq C\Big(\norm{\nabla \u}_{L^{\infty}(D_2)} + \norm{p}_{L^{\infty}(D_2)}\Big)\onenorm{\nabla(\g_{1,h}-\g_1)}\\
	&\quad + C(\rr-r)^{-3/2}\Big(\twonorm{\nabla \u} + \twonorm{p}\Big)\twonorm{\sigma^{3/2}\nabla(\g_{1,h}-\g_1)}.
	\end{align}
	The result then follows from \Cref{proposition:g1_l1} and \Cref{corollary:g1_sigma}.
\end{proof}

\subsection{Estimates for \texorpdfstring{$L^{\infty}(\Omega)$}{L(infty)(Omega)}}

For this case we use the stability of the Ritz projection in $L^{\infty}(\Omega)$ norm as shown in \cite{MR3470741}.
\subsubsection{Regularized Green's function}
This time we define the approximate Green's function $(\g_0,\lambda_0)\in H^1_0(\Omega)^3\times L^2_0(\Omega)$ as the solution to
\begin{subequations}
	\begin{alignat}{3}
	-\Delta \g_0 + \nabla \lambda_0 &= \delta_h \vec e_i &&\quad \text{ in } \Omega, \label{eq:def_stoke_green0_1}\\
	\nabla \cdot \g_0 &= 0 &&\quad \text{ in } \Omega,\label{eq:def_stoke_green0_2} \\
	\g_0 &= \vec 0 &&\quad \text{ on } \partial \Omega.\label{eq:def_stoke_green0_3}
	\end{alignat}
\end{subequations}
Here, $\vec e_i$ is as before the $i$-th standard basis vector in $\R^3$. We also define the finite element approximation $(\g_{0,h},\lambda_{0,h}) \in \V_h \times M_h$ by
\begin{equation}
a((\g_0 - \g_{0,h},\lambda_0-\lambda_{0,h}),(\v_h,q_h)) = 0 \quad \forall (\v_h,q_h) \in \V_h \times M_h.\label{eq:def_stoke_green0_fem}
\end{equation}
Compared to \cref{eq:def_stoke_green1_1,eq:def_stoke_green1_2,eq:def_stoke_green1_3}, the right-hand side of \cref{eq:def_stoke_green0_1} is less singular, which means we can expect faster convergence.
\subsubsection{Auxiliary results for \texorpdfstring{{$(\g_0,\lambda_0)$, $(\g_{0,h},\lambda_{0,h})$}}{(g0,lambda0), (g0h,lambda0h)} and the Ritz projection}
Similarly to the $W^{1,\infty}$ case, we need certain error estimates for the discretization of the regularized Green's function $(\g_0,\lambda_0)$.
However in contrast to $(\g_1,\lambda_1)$, we could not locate such results in the literature. For our purpose we need to establish 
the following results, for which the proofs are given in \Cref{section:duality_interpolation}.
\begin{lemma}
	\label{theorem:g0_l1}
	Let $(\g_0,\lambda_0)$ be the solution of \cref{eq:def_stoke_green0_1,eq:def_stoke_green0_2,eq:def_stoke_green0_3} and $(\g_{0,h},\lambda_{0,h})$ the respective discrete solution. Then, it holds
	\begin{equation}
	\onenorm{\nabla(\g_0-\g_{0,h})} \leq Ch \lnhh.
	\end{equation}
\end{lemma}
The weighted norm estimate follows essentially from Lemma \ref{theorem:g0_l1}. 
\begin{corollary}
	\label{corollary:g0_convergence}
	Let $(\g_0,\lambda_0)$ be the solution of \cref{eq:def_stoke_green0_1,eq:def_stoke_green0_2,eq:def_stoke_green0_3} and $(\g_{0,h},\lambda_{0,h})$ the respective discrete solution. Then, it holds
	\begin{equation}
	\twonorm{\sigma^{3/2}\nabla(\g_0-\g_{0,h})} \leq Ch \lnhh.
	\end{equation}
\end{corollary}

As mentioned before, the proof is based on local and global max-norm estimates for the Ritz projection $R_h\vec z$ of $\vec z\in H^1_0(\Omega)^3$ which is given by
\begin{equation}
(\nabla R_h \vec z, \nabla \v_h) = (\nabla \vec z, \nabla \v_h) \quad \forall \v_h \in \V_h.
\end{equation}
We state the slightly modified results \cite[Theorem 12]{MR3470741} and \cite[Theorem 4.4]{MR3614014} for the convenience of the reader.
\begin{proposition}
	\label{proposition:ritz_projection}
	There exists a constant $C$ independent of $h$ such that, for $\vec z \in H^1_0(\Omega)^3\cap L^{\infty}(\Omega)^3$ the solution of the Laplace equation, it holds that
	\begin{equation}
	\inftynorm{R_h\vec z} \leq C\abs{\ln h} \inftynorm{\vec z}.
	\end{equation}
\end{proposition}
\begin{proposition}
	\label{proposition:local_ritz_projection}
	Let $D \subset D_d \subset \Omega$, where $D_d=\{x\in \Omega : \text{dist}(x,D)\leq d\}$. Then, for $\vec z \in H^1_0(\Omega)^3\cap L^{\infty}(\Omega)^3$ the solution of the Laplace equation, there exists a constant $C$, independent of $h$, such that
	\begin{equation}
	\norm{R_h \vec z}_{L^{\infty}(D)} \leq \lnhh \norm{\vec z}_{L^{\infty}(D_d)} + C_{d}h\norm{\vec z}_{H^1(\Omega)},
	\end{equation}
	where $C_{d}\sim d^{-3/2}$.
\end{proposition}
We will also require the following result.
\begin{lemma}
	\label{lemma:lambda0}
	Let $(\g_0,\lambda_0)$ be the solution of \cref{eq:def_stoke_green0_1,eq:def_stoke_green0_2,eq:def_stoke_green0_3}. Then, it holds
	\begin{equation}
	\onenorm{\nabla \lambda_0} \leq C\lnhh^{1/2} \twonorm{\sigma^{3/2}\nabla \lambda_0} \leq C\lnhh.
	\end{equation}
\end{lemma}
The respective proof is given in \Cref{section:duality_interpolation}.

\subsubsection{Max-norm estimate}
With these tools at hand, we can go ahead with the proof of the theorem.
\begin{proof}[Proof of \Cref{theorem:Linfty_nonzerodivergence} (velocity)]
	We make the ansatz for $\x_0 \in \bar \Omega$
	\begin{align}
	\u_{h,i}(\x_0) = a((\u_h,p_h),(\g_{0,h},\lambda_{0,h}))
	&= a((\u,p),(\g_{0,h},\lambda_{0,h})) \tag{by orthogonality} \\
	&= (\nabla \u, \nabla \g_{0,h}) - (p, \nabla \cdot \g_{0,h}).
	\end{align}
	Since $\g_{0,h} \in \V_h$ we have $(\nabla \u, \nabla \g_{0,h}) = (\nabla R_h\u,\nabla \g_{0,h})$
	and hence by using $\nabla \cdot \g_0 =0$
	\begin{equation}
	\u_{h,i}(\x_0) = (\nabla R_h \u, \nabla \g_{0,h}) - (p, \nabla \cdot \g_{0,h})= (\nabla R_h \u, \nabla \g_{0,h}) - (p, \nabla \cdot (\g_{0,h} - \g_0)) \label{eq:intermediate_step1}.
	\end{equation}
	We can use an inverse estimate on $\nabla R_h \u$. Thus,
	\begin{align}
	(\nabla R_h \u, \nabla \g_{0,h}) &= (\nabla R_h \u, \nabla (\g_{0,h}-\g_0)) - (R_h \u, \Delta \g_0) \\
	&= (\nabla R_h \u, \nabla (\g_{0,h}-\g_0)) - (R_h \u, -\delta_h\vec{e_i} + \nabla \lambda_0 )\\
	&\leq h^{-1}\inftynorm{R_h \u}\onenorm{\nabla (\g_{0,h}-\g_0)} \\
	&\quad 
	+ C\inftynorm{R_h \u}\big(1 + \onenorm{\nabla \lambda_0}\big).
	\end{align}
	For the second term, we get by estimating the divergence by the gradient:
	\begin{equation}
	(p, \nabla \cdot (\g_{0,h} - \g_0)) \leq C \inftynorm{p} \onenorm{\nabla (\g_{0,h} -\g_0)}.
	\end{equation}
	Now we can apply our auxiliary result for $\onenorm{\nabla  (\g_{0,h} - \g_0)}$. Thus, we have by \Cref{theorem:g0_l1} combined with \Cref{proposition:ritz_projection} and \Cref{lemma:lambda0}
	\begin{align}
	\abs{\u_{h,i}(\x_0)} &\leq C\abs{\ln h}\inftynorm{\u}h^{-1}\onenorm{\nabla (\g_{0,h}-\g_0)}+\inftynorm{p} \onenorm{\nabla (\g_{0,h} -\g_0)}\\
	&\leq C\Big(\abs{\ln h}^2\inftynorm{\u}+ \lnhh h\inftynorm{p}\Big).
	\end{align}
\end{proof}

\subsubsection{Localization}
The approach for the localization in the $L^{\infty}$ case is similar to  $W^{1,\infty}$ but different in the sense that we again use the stability of $R_h$ in $L^\infty$ norm.
\begin{proof}[Proof of \Cref{theorem:Linfty_localization} (velocity)]	
	We only consider $\x_0 \in T_{\x_0}\subset D_1$.
	As before, using \cref{eq:def_stoke_green0_fem,eq:def_stoke_fem,eq:def_stoke_fem_equation} gives
	\begin{align}
	\u_{h,i}(\x_0) &= a((\u_h,p_h),(\g_{0,h},\lambda_{0,h})) 
	= a((\u,p),(\g_{0,h},\lambda_{0,h})) \tag{by orthogonality} \\
	&= (\nabla \u, \nabla \g_{0,h}) - (p, \nabla \cdot \g_{0,h})
	:= I_1 + I_2.
	\end{align}
	Using the properties of the Ritz projection we first consider 
	\begin{align}
	I_1 &= (\nabla R_h\u, \nabla \g_{0,h})\\
	&= (\nabla R_h\u, \nabla \g_{0}) + (\nabla R_h\u, \nabla (\g_{0,h}- \g_{0}))\\
	&= -(R_h\u, \Delta \g_{0}) + (\nabla R_h\u, \nabla (\g_{0,h}- \g_{0}))\\
	&= (R_h\u,\delta_h\vec{e_i} - \nabla \lambda_0) + (\nabla R_h\u, \nabla (\g_{0,h}- \g_{0}))
	\end{align}
	Next, we apply \cref{eq:def_delta} and split the domain into $D_2$ and $\Omega \backslash D_2$
	\begin{align}
	I_1&\leq \norm{R_h\u}_{L^{\infty}(T_{\x_0})} + \norm{R_h\u}_{L^{\infty}(D_2)}\onenorm{\nabla \lambda_0} + \norm{\nabla R_h\u}_{L^{\infty}(D_2)}\onenorm{\nabla(\g_{0,h}- \g_{0})}\\
	&\quad + \norm{\sigma^{-3/2}R_h\u}_{L^2(\Omega\backslash D_2)}\twonorm{\sigma^{3/2}\nabla\lambda_0} \\
	&\quad + \norm{\sigma^{-3/2}\nabla R_h\u}_{L^2(\Omega\backslash D_2)}\twonorm{\sigma^{3/2}\nabla(\g_{0,h}- \g_{0})}.
	\end{align}
	Using the properties of $\sigma$ and applying an inverse inequality gives
	\begin{align}
	I_1&\leq C\norm{ R_h\u}_{L^{\infty}(D_2)}\big(1 + \onenorm{\nabla \lambda_0} +h^{-1}\onenorm{\nabla(\g_{0,h}- \g_{0})}\big)\\
	&\quad + C_d \twonorm{R_h\u}\big(\twonorm{\sigma^{3/2}\nabla\lambda_0} + h^{-1}\twonorm{\sigma^{3/2}\nabla(\g_{0,h}- \g_{0})}\big).
	\end{align}
	To estimate $R_h\u$ in the $L^{\infty}$ and $L^2$ norm we can apply \Cref{proposition:local_ritz_projection} and an estimate for $\twonorm{R_h\u - \u}$ %
	to see together with \Cref{theorem:g0_l1}, \Cref{corollary:g0_convergence} and \Cref{lemma:lambda0} that
	\begin{align}
	I_1 &\leq C\lnhh \norm{\u}_{L^{\infty}(D_2)}(1+\lnhh) + C_d \lnhh \Big(\twonorm{\u} + h \norm{\u}_{H^1(\Omega)}\Big)\\
	&\leq C_d\lnhh^2 \norm{\u}_{L^{\infty}(D_2)} + C_d \lnhh \Big(\twonorm{\u} + h \norm{\u}_{H^1(\Omega)}\Big).
	\end{align}
	Using similar arguments we get for
	\begin{align}
	I_2 &= - (p, \nabla \cdot (\g_{0,h}-\g_0))\\
	&\leq C\norm{p}_{L^{\infty}(D_2)}\onenorm{\nabla(\g_{0,h}- \g_{0})}
	+ C_d \twonorm{p}\twonorm{\sigma^{3/2}\nabla(\g_{0,h}- \g_{0})}\\
	&\leq C\lnhh \norm{p}_{L^{\infty}(D_2)} + C_d \lnhh\twonorm{p},
	\end{align}
	which concludes the proof of the theorem.
\end{proof}

\section{Estimates for the regularized Green's function}
\label{section:duality_interpolation}
In this section we prove \Cref{corollary:g1_sigma,corollary:g0_convergence} and \Cref{lemma:lambda0,theorem:g0_l1} which we need in order to establish the main theorems.

\subsection{Dyadic decomposition}
For the proof of our results, we use a dyadic decomposition of the domain $\Omega$, which we will introduce next. Without loss of generality, we assume that the diameter of $\Omega$ is less than $1$. We put $d_j=2^{-j}$ and consider the decomposition $\Omega = \Omega_* \cup \bigcup_{j=0}^{J}\Omega_j$, where 
\begin{equation}
\Omega_* = \{\x \in \Omega : \abs{\x - \x_0}\leq K h \},\qquad
\Omega_j = \{\x \in \Omega : d_{j+1} \leq \abs{\x - \x_0} \leq d_j \},
\end{equation}
$K$ is a sufficiently large constant to be chosen later and $J$ is an integer such that 
\begin{equation}
2^{-(J+1)}\leq K h \leq 2^{-J}.\label{eq:dydic_J}
\end{equation} 
We keep track of the explicit dependence on $K$.
Furthermore, we consider the following enlargements of $\Omega_j$
\begin{align}
\Omega_j' &= \{\x \in \Omega : d_{j+2} \leq \abs{\x -\x_0} \leq d_{j-1} \},\\
\Omega_j'' &= \{\x \in \Omega : d_{j+3} \leq \abs{\x -\x_0} \leq d_{j-2} \},\\
\Omega_j''' &= \{\x \in \Omega : d_{j+4} \leq \abs{\x -\x_0} \leq d_{j-3} \}.
\end{align}

\begin{lemma}\label{lemma: Greens function dyadic}
	There exists a constant $C$ independent of $d_j$ such that for any $\x\in \Om_j$,
	$$
	|\nabla \g_0(\x)|+d_j^{-1}| \g_0(\x)|+|\lambda_0(\x)|\le Cd_j^{-2}.
	$$
\end{lemma}
\begin{proof}
	Due to \cref{eq:greens_matrix_2} and \Cref{proposition:Green_matrix}, it holds for $\x \in \Omega_j$
	\begin{align}
	\abs{ \lambda_0(\x)} &= \Big\vert\int_{\Omega}  G_{4}(\x,\y)\cdot \delta_h(\y)\vec e_i d\y\Big\vert
	\leq \int_{T_{\x_0}} \abs{ G_{i,4}(\x,\y)}\abs{\delta_h(\y)}d\y\\
	&\leq C \int_{T_{\x_0}} \frac{\abs{\delta_h (\y)}}{\abs{\x-\y}^2}d\y\leq C d^{-2}_j \onenorm{\delta_h} \leq Cd_j^{-2},
	\end{align}
	where we used that $dist(x_0,\Om_j)\geq Cd_j$.
	Similarly,  without loss of generality, considering the $k$-th component, $1\leq k \leq 3$, we have for
	\begin{align}
	\abs{\partial_x \g_{0,k}(\x)} 
	= \Big\vert\int_{\Omega} \partial_x G_{k}(\x,\y)\cdot \delta_h (\y)\vec e_i d\y\Big\vert
	&\leq \int_{T_{\x_0}} \abs{\partial_x G_{i,k}(\x,\y)}\abs{\delta_h (\y)} d\y\\
	&\leq \int_{T_{\x_0}} \frac{\abs{\delta_h(\y)}}{\abs{\x-\y}^{2}}d\y \leq Cd_j^{-2}.
	\end{align}
	The estimate for $\g_{0,k}(\x)$ is similar. 
\end{proof}
As an immediate application of the above result and \Cref{corollary:localh2} we obtain the following result. 
\begin{corollary}\label{cor: g0_and_lambda_0}
	$$
	\norm{\vec{g}_0}_{H^2(\Omega_j)}+\norm{\nabla\lambda_0}_{L^2(\Omega_j)}\le Cd_j^{-3/2}.
	$$
\end{corollary}
\begin{proof}
	By \Cref{corollary:localh2}, the H\"older  estimates, and Lemma \ref{lemma: Greens function dyadic} (with $\Om'_j$ instead of $\Om_j$), we obtain
	$$
	\begin{aligned}
	\norm{\g_0}_{H^2(\Omega_j)}+	\norm{\nabla\lambda_0}_{L^2(\Omega_j)} &\le Cd_j^{-1}\left(\norm{\lambda_0}_{L^2(\Omega'_j)}+\norm{\nabla \g_0}_{L^2(\Omega'_j)}+d_j^{-1}\norm{ \g_0}_{L^2(\Omega'_j)}\right)\\
	&\le	Cd_j^{1/2}\left(\norm{\lambda_0}_{L^\infty(\Omega'_j)}+\norm{\nabla \g_0}_{L^\infty(\Omega'_j)}+d_j^{-1}\norm{ \g_0}_{L^\infty(\Omega'_j)}\right)\\
	&\le Cd_j^{-3/2}.
	\end{aligned}
	$$
	
\end{proof}

\subsection{\texorpdfstring{$L^1(\Omega)$}{L(1)(Omega)} interpolation estimate for \texorpdfstring{$\lambda_0$}{lambda0}}
\begin{theorem}
	\label{theorem:lambda_0_interpolation}
	For $(\g_0,\lambda_0)$ the solution of \cref{eq:def_stoke_green0_1,eq:def_stoke_green0_2,eq:def_stoke_green0_3}, it holds
	\begin{equation}
	\onenorm{\lambda_0 - r_h(\lambda_0)} \leq C h \lnhh.
	\end{equation}
\end{theorem}
\begin{proof}
	Using the dyadic decomposition and the Cauchy-Schwarz inequality
	\begin{align}
	\onenorm{\lambda_0 - r_h (\lambda_0)} 
	&\leq \norm{\lambda_0 - r_h (\lambda_0)}_{L^1(\Omega_*)} + \sum_{j=1}^{J}\norm{\lambda_0 - r_h (\lambda_0)}_{L^1(\Omega_j)} \\
	&\leq  (K h)^{3/2}\norm{\lambda_0 - r_h (\lambda_0)}_{L^2(\Omega_*)} + C\sum_{j=1}^J d_j^{3/2}\norm{\lambda_0 - r_h (\lambda_0)}_{L^2(\Omega_j)}. \label{eq:lambda_l1_interpolation_1}
	\end{align} 
	We apply \Cref{assumption:approximation} and the $H^2$ regularity as in \cref{eq:H2_estimate}, which give
	\begin{equation}
	\twonorm{\lambda_0 - r_h (\lambda_0)} \leq Ch \twonorm{\nabla \lambda_0} \leq Ch \twonorm{\delta_h} \leq Ch^{-1/2}.
	\end{equation}	
	This implies for the first term in \cref{eq:lambda_l1_interpolation_1}
	\begin{equation}
	(K h)^{3/2}\norm{\lambda_0 - r_h (\lambda_0)}_{L^2(\Omega_*)} 
	\leq CK^{3/2}h.
	\end{equation}
	For the second term, by the approximation estimate \Cref{assumption:approximation} and \Cref{cor: g0_and_lambda_0} it follows
	\begin{equation}
	\norm{\lambda_0 - r_h (\lambda_0)}_{L^2(\Omega_j)} \leq Ch \norm{\nabla\lambda_0}_{L^2(\Omega'_j)} \le Chd_j^{-3/2}. \label{eq:lambda0_est}
	\end{equation}
	Hence, we can conclude
	\begin{equation}
	\sum_{j=1}^J d_j^{3/2}\norm{\lambda_0 - r_h (\lambda_0)}_{L^2(\Omega_j)} \leq \sum_{j=1}^{J}Ch \leq Ch J.
	\end{equation}
	From \cref{eq:dydic_J}, we see that $J$ scales logarithmically in $h$ and thus get the claimed result.
\end{proof}
\subsection{Local duality argument}
In the following theorem, we again consider the sub-domains $\Omega_j$ from the dyadic decomposition in a duality argument.
For the error
\begin{equation}
\norm{\g_0-\g_{0,h}}_{L^2(\Omega_j')} = 
\sup_{\substack{\twonorm{\v}\leq 1 \\ \v \in C^{\infty}_0(\Omega_j')}} (\g_0 - \g_{0,h},\v)
\end{equation}
we can make a duality argument using the dual problem
	\begin{equation}
	-\Delta \w + \nabla \varphi = \v \quad \text{in } \Omega, \label{eq:dual_g0_1}\quad
	\nabla \cdot \w =0 \quad \text{in } \Omega, %
	\quad \w =0 \quad \text{on } \partial \Omega.%
	\end{equation}

\begin{theorem}
	\label{theorem:g0_dual}
	For $(\g_0,\lambda_0)$ the solution of \cref{eq:def_stoke_green0_1,eq:def_stoke_green0_2,eq:def_stoke_green0_3} and $\alpha \in (0,1)$ it holds
	\begin{align}
	\norm{\g_0-\g_{0,h}}_{L^2(\Omega_j')} 
	\leq C h \norm{\nabla(\g_0-\g_{0,h})}_{L^2(\Omega_j''')} + Ch^{\alpha}d_j^{-1/2 - \alpha} \onenorm{\nabla (\g_0-\g_{0,h})} \\
	+ Ch^{1+\alpha}d_j^{-1/2-\alpha}\lnhh .
	\end{align}
\end{theorem}
\begin{proof}
	By using \cref{eq:dual_g0_1} and that $\g_0$ and $\g_{h,0}$ are divergence free for $r_h(\varphi)$, the bilinear form $a(\cdot,\cdot)$ from \cref{eq:def_stoke_fem} and \Cref{assumption:divergence}, it follows
	\begin{align}
	(\g_0 -\g_{0,h},\v) 
	&= (\nabla(\g_0-\g_{0,h}),\nabla \w) - (\varphi,\nabla \cdot (\g_0 -\g_{0,h}))\\
	&= (\nabla(\g_0 - \g_{0,h}), \nabla(\w - P_h(\w))) \\
	&\quad + (\nabla(\g_0 - \g_{0,h}), \nabla P_h(\w)) -(\varphi -r_h(\varphi),\nabla \cdot (\g_0 -\g_{0,h}))\\
	&=(\nabla(\g_0 - \g_{0,h}), \nabla(\w - P_h(\w))) \\
	&\quad + (\lambda_0 - \lambda_{0,h}, \nabla \cdot P_h(\w)) -(\varphi -r_h(\varphi),\nabla \cdot (\g_0 -\g_{0,h}))\\
	&=(\nabla(\g_0 - \g_{0,h}), \nabla(\w - P_h(\w))) \\
	&\quad + (\lambda_0 - r_h(\lambda_0), \nabla \cdot  (P_h(\w)-\w)) -(\varphi -r_h(\varphi),\nabla \cdot (\g_0 -\g_{0,h}))\\
	&:= \tau_1 + \tau_2 + \tau_3.
	\end{align}
	For $\tau_1$, we split the term
	\begin{align}
	\tau_1 &= (\nabla(\g_0 - \g_{0,h}), \nabla(\w -  P_h(\w)))_{\Omega_j'''} + (\nabla(\g_0 - \g_{0,h}), \nabla(\w - P_h(\w)))_{\Omega \backslash\Omega_j'''}\\
	&:= \tau_{11} + \tau_{12}.
	\end{align}
	We then can estimate $\tau_{11}$ using \Cref{assumption:approximation} for $ P_h$
	\begin{align}
	\tau_{11} &\leq \norm{\nabla(\g_0 -\g_{0,h})}_{L^2(\Omega_j''')} \norm{\nabla(\w - P_h(\w))}_{L^2(\Omega)}\\&\leq Ch\norm{\nabla(\g_0 -\g_{0,h})}_{L^2(\Omega_j''')}\norm{\w}_{H^2(\Omega)} 
	\leq Ch\norm{\nabla(\g_0 -\g_{0,h})}_{L^2(\Omega_j''')}.
	\end{align}
	Now we use \cite[(5.11)]{MR2945141} and \Cref{assumption:approximation_hoelder} to see that
	\begin{equation}
	\tau_{12} \leq C h^{\alpha}\onenorm{\nabla(\g_0 - \g_{0,h})} \norm{\w}_{C^{1+\alpha}(\Omega \backslash\Omega_j'')} \leq C h^{\alpha}d_j^{-1/2 - \alpha}\onenorm{\nabla(\g_0 - \g_{0,h})}.
	\end{equation}
	Analogously, we split $\tau_2$
	\begin{align}
	\tau_2 &= -(\lambda_0 -r_h(\lambda_0), \nabla \cdot(\w -  P_h(\w))_{\Omega_j'''} -(\lambda_0 -r_h(\lambda_0), \nabla \cdot(\w -  P_h(\w))_{\Omega \backslash\Omega_j'''}\\
	&:= \tau_{21} + \tau_{22}.
	\end{align}
	Then again, we use approximation results and Corollary \ref{cor: g0_and_lambda_0}, to see
	\begin{equation}
	\tau_{21} \leq C h^2\norm{\nabla \lambda_0}_{L^2(\Omega_j'')}\norm{\w}_{H^2(\Omega)}\leq C h^2\norm{\nabla \lambda_0}_{L^2(\Omega_j'')}\le C h^2d_j^{-3/2}.
	\end{equation}
	For the second term, we apply again the H\"older estimate, \Cref{theorem:lambda_0_interpolation} and \cite[(5.11)]{MR2945141}
	\begin{multline}
	\tau_{22} \leq \onenorm{\lambda_0- r_h(\lambda_0)}\norm{\nabla(\w -  P_h(\w))}_{L^{\infty}(\Omega \backslash\Omega_j''')} \\
	\leq Ch^{1+\alpha} \lnhh \norm{\w}_{C^{1+\alpha}(\Omega\backslash\Omega_j'')} \leq Ch^{1+\alpha}d_j^{-1/2 -\alpha}\lnhh.
	\end{multline}
	It remains to deal with $\tau_3$, we split again
	\begin{equation}
	\tau_3 \leq \abs{(\varphi - r_h(\varphi),\nabla \cdot (\g_0- \g_{0,h}))_{\Omega_j'''}} +\abs{(\varphi - r_h(\varphi),\nabla \cdot (\g_0- \g_{0,h}))_{\Omega \backslash \Omega_j'''}} := \tau_{31} + \tau_{32}.
	\end{equation}
	Analogously to before, we estimate
	\begin{align}
	\tau_{31} &\leq \norm{\varphi - r_h(\varphi)}_{L^2(\Omega_j''')} \norm{\nabla(\g_0-\g_{0,h})}_{L^2(\Omega_j''')} \leq Ch \norm{\nabla(\g_0-\g_{0,h})}_{L^2(\Omega_j''')} \quad \text{and}\\
	\tau_{32} &\leq \norm{\varphi - r_h(\varphi)}_{L^{\infty}(\Omega\backslash \Omega_j''')} \onenorm{\nabla(\g_0-\g_{0,h})} \leq Ch^{\alpha} d_j^{-1/2 - \alpha} \onenorm{\nabla(\g_0-\g_{0,h})}.
	\end{align}
	The estimate for $\norm{\varphi - r_h(\varphi)}_{L^{\infty}(\Omega\backslash \Omega_j''')}$ is given in \cite[p. 17]{MR2945141}.
	Summing up, we have
	\begin{multline}
	\norm{\g_0-\g_{0,h}}_{L^2(\Omega_j)} 
	\leq Ch\norm{\nabla(\g_0 -\g_{0,h})}_{L^2(\Omega_j''')} + C h^{\alpha}d_j^{-1/2 - \alpha}\onenorm{\nabla(\g_0-\g_{0,h})} \\\nonumber
	+ h^2 d_j^{-3/2}  + Ch^{1+\alpha}d_j^{-1/2 -\alpha}\lnhh.
	\end{multline}
	Now, because $h \leq d_j$ due to \cref{eq:dydic_J} and $\alpha \leq 1$, it holds $h^2d_j^{-3/2} \leq h^{1+\alpha} d_j^{-1/2 - \alpha}$.
	Thus, we arrive at the conclusion of the theorem.
\end{proof}

\subsection{\texorpdfstring{$L^1(\Omega)$}{L(1)(Omega)} estimate and weighted estimate}

Now we can proceed with the proof of \Cref{theorem:g0_l1}.
\begin{proof}[Proof of \Cref{theorem:g0_l1}]
	We again use the dyadic decomposition and the Cauchy-Schwarz inequality to see
	\begin{align}
	\onenorm{\nabla(\g_0-&\g_{0,h})} 
	\leq \norm{\nabla(\g_0 - \g_{0,h})}_{L^1(\Omega_*)} + \sum_{j=1}^{J}\norm{\nabla (\g_0-\g_{0,h})}_{L^1(\Omega_j)} \\
	&\leq (Kh)^{3/2}\twonorm{\nabla(\g_0-\g_{0,h}} + C\sum_{j=1}^{J}d_j^{3/2}\norm{\nabla (\g_0-\g_{0,h})}_{L^2(\Omega_j)}.
	\label{eq:theorem_g0_l1_6}
	\end{align}
	Applying \Cref{proposition:fem_convergence}, \Cref{assumption:approximation}, $H^2$ regularity as stated in \cref{eq:H2_estimate} and \cref{eq:deltah_est} leads to the following estimate for the first term
	\begin{align}
	h^{3/2}\twonorm{\nabla(\g_0-\g_{0,h})} 
	&\leq Ch^{5/2} \Big(\norm{\g_0}_{H^2(\Omega)}+\norm{\lambda_0}_{H^1(\Omega)}\Big)\label{eq:theorem_g0_l1_1}\\
	&\leq C h^{5/2} \norm{\delta_h}_{L^2(T_{\x_0})}\leq Ch.
	\end{align}
	In the following, we consider the second term for which we want to apply the local energy estimate from \Cref{proposition:local_energy_estimate}:
	\begin{align}
	\norm{\nabla(\g_0-\g_{0,h})}_{L^2(\Omega_j)} 
	&\leq C\big(\norm{\nabla(\g_0 - P_h(\g_0))}_{L^2(\Omega_j')} + \norm{\lambda_0 - r_h(\lambda_0)}_{L^2(\Omega_j')}\Big) \\
	&\quad + C(\varepsilon d_j)^{-1} \norm{\g_0- P_h(\g_0)}_{L^2(\Omega_j')}
	+ \varepsilon\norm{\nabla(\g_0 - \g_{0,h})}_{L^2(\Omega_j')} \\
	&\quad +  C(\varepsilon d_j)^{-1}\norm{\g_0 - \g_{0,h}}_{L^2(\Omega_j')}. \label{eq:theorem_g0_l1_4}
	\end{align}
	For the first two terms we use approximation results and Corollary \ref{cor: g0_and_lambda_0}, to obtain
	\begin{align}
	\norm{\nabla(\g_0 - P_h(\g_0))}_{L^2(\Omega_j')} + \norm{\lambda_0 - r_h(\lambda_0)}_{L^2(\Omega_j')}
	&\leq Ch \Big(\norm{\g_0}_{H^2(\Omega_j'')} + \norm{\lambda_0}_{H^1(\Omega_j'')}\Big)\\
	&\leq Chd_j^{-3/2} .
	\end{align}
	The contribution to the sum is given by
	\begin{equation}
	\sum_{j=1}^{J}d_j^{3/2}(\norm{\nabla(\g_0 - P_h(\g_0))}_{L^2(\Omega_j')} + \norm{\lambda_0 - r_h(\lambda_0)}_{L^2(\Omega_j')}) \leq Ch J \leq Ch \lnhh,\label{eq:summing_up_J}
	\end{equation}
	where due to \cref{eq:dydic_J} we see that $J \sim \lnhh$.
	Similarly, we see
	\begin{equation}
	(\varepsilon d_j)^{-1}\norm{\g_0 -  P_h(\g_0)}_{L^2(\Omega_j')} \leq C\frac{h}{\varepsilon d_j}h d_j^{-3/2}. \label{eq:theorem_g0_l1_3}
	\end{equation}
	For $\alpha>0$, it holds
	\begin{equation}
	\sum_{j=1}^{J}\left(\frac{h}{d_j} \right)^{\alpha} \leq h^{\alpha}\sum_{j=1}^{J} 2^{j\alpha}\leq Ch^{\alpha}2^{\alpha J} \leq CK^{-\alpha}.\label{eq:theorem_g0_l1_5}
	\end{equation}
	Thus, we get by summing up \cref{eq:theorem_g0_l1_3} and using \cref{eq:theorem_g0_l1_5} with $\alpha =1$ that $\sum_{j=1}^{J}C\frac{h}{\varepsilon d_j}h\leq C(K\varepsilon)^{-1}h$.
	To summarize our results so far, we define $M_j = d_j^{3/2}\norm{\nabla (\g_0-\g_{0,h})}_{L^2(\Omega_j)}$, $M'_j = d_j^{3/2}\norm{\nabla (\g_0-\g_{0,h})}_{L^2(\Omega'_j)}$ and substitute into \cref{eq:theorem_g0_l1_4}
	\begin{equation}
	\sum_{j=1}^{J}M_j \leq Ch\lnhh + C(K\varepsilon)^{-1}h  + \varepsilon\sum_{j=1}^{J}M'_j +  C\sum_{j=1}^{J}(\varepsilon d_j)^{-1}d_j^{3/2}\norm{\g_0 - \g_{0,h}}_{L^2(\Omega_j')}.
	\end{equation}
	Next, we apply \Cref{theorem:g0_dual} to the last term
	\begin{multline}
	\sum_{j=1}^{J}M_j 
	\leq Ch\lnhh + C(K\varepsilon)^{-1}h  + \varepsilon\sum_{j=1}^{J}M'_j\\
	+ C\varepsilon^{-1}\sum_{j=1}^{J}\bigg(
	d_j^{1/2}h\norm{\nabla(\g_0-\g_{0,h})}_{L^2(\Omega_j''')} + \bigg[\frac{h}{d_j}\bigg]^{\alpha} \onenorm{\nabla (\g_0-\g_{0,h})} + h\bigg[\frac{h}{d_j}\bigg]^{\alpha}\lnhh\bigg). \nonumber
	\end{multline}
	We expand the sum over the last three terms so that we get
	\begin{multline}
	\sum_{j=1}^{J}M_j \nonumber
	\leq C\Big(h\lnhh + (K\varepsilon)^{-1}h  + \varepsilon\sum_{j=1}^{J}M'_j
	+ \frac{h}{d_J}\varepsilon^{-1}\sum_{j=1}^{J}
	d_j^{3/2}\norm{\nabla(\g_0-\g_{0,h})}_{L^2(\Omega_j''')}\Big) \\+C\varepsilon^{-1}\sum_{j=1}^{J}\left[\frac{h}{d_j}\right]^{\alpha} \onenorm{\nabla (\g_0-\g_{0,h})} 
	+ Ch\varepsilon^{-1}\sum_{j=1}^{J}\left[\frac{h}{d_j}\right]^{\alpha}\lnhh.
	\end{multline}
	Now we can again use \cref{eq:theorem_g0_l1_5} on the last two summands to arrive at
	\begin{multline}
	\sum_{j=1}^{J}M_j\leq Ch\lnhh  + C\varepsilon\sum_{j=1}^{J}M'_j  + \nonumber CK^{-\alpha}\varepsilon^{-1}\Big(\onenorm{\nabla (\g_0-\g_{0,h})}+ h\lnhh\Big) \\
	+ C(K\varepsilon)^{-1}\sum_{j=1}^{J}
	d_j^{3/2}\norm{\nabla(\g_0-\g_{0,h})}_{L^2(\Omega_j''')},
	\end{multline}
	where we also used that $h/d_J \leq K^{-1}$ and $K>1$. Now for the second and last term, we easily see
	\begin{equation}
	\sum_{j=1}^{J}M'_j + \sum_{j=1}^{J}
	d_j^{3/2}\norm{\nabla(\g_0-\g_{0,h})}_{L^2(\Omega_j''')}
	\leq C \sum_{j=1}^{J}M_j + C(Kh)^{3/2}\norm{\nabla(\g_0-\g_{0,h}}_{L^2(\Omega_*)},%
	\end{equation}
	where the last term is again bounded by $CK^{3/2}h$.
	Combined, this means we have for constant $K>1$ and $\varepsilon>0$
	\begin{multline}
	\sum_{j=1}^{J}M_j \nonumber
	\leq Ch\lnhh  + C((K\varepsilon)^{-1} + \varepsilon)\sum_{j=1}^{J}M_j + C K^{3/2} \varepsilon h + CK^{1/2}\varepsilon^{-1}h \\
	+ CK^{-\alpha}\varepsilon^{-1}\Big(\onenorm{\nabla (\g_0-\g_{0,h})} + h\lnhh \Big).
	\end{multline}
	We make $C\varepsilon<1/4$ and $C(K\varepsilon)^{-1}<1/4$ by choosing $\varepsilon$ small and $K$ big enough. After kicking back the sum to the left-hand side this leads to
	\begin{equation}
	\sum_{j=1}^{J}M_j 
	\leq C_{K,\varepsilon}h\lnhh + CK^{-\alpha}\varepsilon^{-1}\onenorm{\nabla (\g_0-\g_{0,h})}.
	\end{equation}
	We now treat $\varepsilon$ as a constant. Finally substituting this into \cref{eq:theorem_g0_l1_6} 
	\begin{equation}
	\onenorm{\nabla(\g_0-\g_{0,h})} \leq C_{K,\varepsilon}h\lnhh  + CK^{-\alpha}\onenorm{\nabla (\g_0-\g_{0,h})}\label{eq:theorem_g0_l1_7}
	\end{equation}
	and choosing $K$ large enough such that $CK^{-\alpha}<1/2$, we get the result.
\end{proof}
As a corollary to the theorem, we get the respective estimate for weighted norms.
\begin{proof}[Proof of \Cref{corollary:g0_convergence}]
	This corollary directly follows using the same techniques as above and the fact $\sigma(\x) \sim d_j$ on $\Omega_j$.
	We start by splitting the left-hand side according to the dyadic decomposition
	\begin{align}
	\twonorm{\sigma^{3/2}\nabla(\g_0-&\g_{0,h})}
	\leq \norm{\sigma^{3/2}\nabla(\g_0-\g_{0,h})}_{L^2(\Omega_*)} + \sum_{j=1}^{J}\norm{\sigma^{3/2}\nabla(\g_0-\g_{0,h})}_{L^2(\Omega_j)}\\
	&\leq C(\kappa h)^{3/2} \norm{\nabla(\g_0-\g_{0,h})}_{L^2(\Omega_*)} + C\sum_{j=1}^{J} d_j^{3/2} \norm{\nabla(\g_0-\g_{0,h})}_{L^2(\Omega_j)}.
	\label{eq:corollary_g0_l1_6}
	\end{align}
	Without loss of generality, we can assume $\kappa = K$. After going through the same steps as in the proof of \Cref{theorem:g0_l1}, particularly \cref{eq:theorem_g0_l1_6}, we end up with the right-hand side of \cref{eq:theorem_g0_l1_7}
	\begin{equation}
	\twonorm{\sigma^{3/2}\nabla(\g_0-\g_{0,h})}
	\leq Ch\lnhh + CK^{-\alpha}\onenorm{\nabla (\g-\g_h)}.
	\end{equation}
	Now applying \Cref{theorem:g0_l1} to estimate $\onenorm{\nabla (\g-\g_h)}$ we arrive at the result.
\end{proof}
Similarly we can conclude the following result.
\begin{proof}[Proof of \Cref{corollary:g1_sigma}]
	Again using the fact $\sigma(\x) \sim d_j$ on $\Omega_j$,
	we start by splitting the left-hand side according to the dyadic decomposition
	\begin{multline}
	\twonorm{\sigma^{3/2}\nabla(\g_1-\g_{1,h})}\nonumber
	\\ \leq C(\kappa h)^{3/2} \norm{\nabla(\g_1-\g_{1,h})}_{L^2(\Omega_*)} + C\sum_{j=1}^{J} d_j^{3/2} \norm{\nabla(\g_1-\g_{1,h})}_{L^2(\Omega_j)}.
	\label{eq:corollary_g1_l1_6}
	\end{multline}
	As before, we can assume $\kappa = K$. This is equal to the term introduced by the dyadic decomposition in the proof of \cite{MR2945141}. Again, following the same steps as there, we get
	\begin{equation}
	\twonorm{\sigma^{3/2}\nabla(\g_1-\g_{1,h})}
	\leq C + C\onenorm{\nabla (\g-\g_h)},
	\end{equation}
	where $C$ depends the constants introduced in the proof of \cite{MR2945141}.
	Nonetheless, applying \Cref{proposition:g1_l1} to estimate $\onenorm{\nabla (\g-\g_h)}$ we arrive at the result.
\end{proof}

\subsection{Proof of \texorpdfstring{\Cref{lemma:lambda0}}{Lemma 3.9}}
\begin{proof}[Proof of \Cref{lemma:lambda0}]
	We use the dyadic decomposition introduced in the beginning of \Cref{section:duality_interpolation} to get the following estimate  due to $\sigma \sim d_j$ on $\Omega_j$ ($\sigma \sim Kh$ on $\Omega_*$)
	\begin{equation}
	\stwonorm{\sigma^{3/2}\nabla\lambda_0} \leq  \nonumber %
	Ch^3\norm{\nabla\lambda_0}_{L^2(\Omega)}^2+\sum_{j=1}^{J}d^3_j\norm{\nabla\lambda_0}_{L^2(\Omega_j)}^2.
	\end{equation}
	The first summand is bounded by a constant $C$ due to \cref{eq:H2_estimate} and \cref{eq:deltah_est}.
	By Corollary \ref{cor: g0_and_lambda_0} we see that $\norm{\nabla\lambda_0}^2_{L^2(\Omega_j)}\le Cd_j^{-3}$ and as a result
	\begin{equation}
	\sum_{j=1}^{J}d^3_j\norm{\nabla\lambda_0}^2_{L^2(\Omega_j)}\le C\sum_{j=1}^{J}1 = CJ \leq C\lnhh.
	\end{equation}
	This proves the result for the weighted case and by $\twonorm{\sigma^{-3/2}}\leq \lnhh^{1/2}$ the $L^1$ estimate.
\end{proof}

\section{Estimates for the pressure}
\label{section:pressure}
We now consider estimates for the remaining component of our Stokes system, the pressure. Similarly to before, let $\delta_h$ denote a smooth delta function on the tetrahedron where the maximum for the pressure is attained.
We may define the following regularized Green's function to deal with the pressure
\begin{equation}
	-\Delta \G + \nabla \Lambda = 0 \quad \text{in } \Omega, \label{eq:green_pressure_1}
	\quad\nabla \cdot \G = \delta_h - \phi \quad \text{in } \Omega, 
	\quad\G = 0 \quad\text{on } \partial \Omega. 
\end{equation}
By construction we have $\int_{\Omega}\delta_h (\x) - \phi(\x) d\x = 0$.
This also allows us to apply similar arguments as in \cite{MR3422453,MR2945141}, only with different bounds for the appearing $\u_h$ terms.

The global case has already been discussed in \cite{MR3422453,MR2945141}, thus we now focus on localized estimates. As before, we need some auxiliary results which we state now.
\begin{proposition}
	\label{proposition:lambda_interpolation}
	\begin{equation}
	\onenorm{\nabla(P_h(\G)-\G)} + \onenorm{r_h(\Lambda)-\Lambda} \leq C.
	\end{equation}
\end{proposition}
A proof of this is given in \cite[Lemma 5.4]{MR2945141}.
The following corollary follows by the same arguments as \Cref{corollary:g1_sigma} and \Cref{corollary:g0_convergence}.
\begin{corollary}
	\label{corollary:lambda_interpolation}
	\begin{equation}
	\twonorm{\sigma^{3/2}\nabla(P_h(\G)-\G)} + \twonorm{\sigma^{3/2}(r_h(\Lambda)-\Lambda)} \leq C.
	\end{equation}
\end{corollary}

\begin{proof}[Proof of \Cref{theorem:W1infty_localization} (pressure)]
	
	For this we again split the domain into $D_2$ and $\Omega\backslash D_2$ and only consider $\x_0\in T_{\x_0} \subset D_1$.

	The pointwise estimate of $p_h$ can be expanded in the following way
	\begin{equation}
	p_h(\x_0) = (p_h,\delta_h) = (p_h, \delta_h - \phi) + (p_h,\phi)=(p_h, \delta_h - \phi) +(p_h - p, \phi) + (p,\phi).
	\end{equation}
	The the last two terms we may estimate using \Cref{proposition:fem_convergence}
	\begin{equation}
	(p_h - p, \phi) + (p,\phi)
	\leq C \twonorm{\phi}\Big(\twonorm{p-p_h} + \twonorm{p}\Big)
	\leq C \Big(\twonorm{\nabla \u} + \twonorm{p}\Big)\label{eq:p_phi_1}.
	\end{equation}
	By assumption $\phi$ is bounded on $\Omega$.
	For the first term, we can see by \Cref{assumption:divergence} that
	\begin{align}
	(p_h, \delta_h - \phi) &= (p_h, \nabla \cdot \G)= (p_h, \nabla \cdot P_h (\G)) \\
	&= (p, \nabla \cdot P_h (\G)) + (p_h -p, \nabla \cdot P_h (\G))
	:= I_1 + I_2.
	\end{align}
	For $I_1$, we get the following estimate
	\begin{align}
	I_1 &= (p, \nabla \cdot (P_h(\G)-\G)) + (p,\delta_h - \phi)\\
	&\leq \norm{p}_{L^{\infty}(D_2)}\Big(\onenorm{\nabla (P_h(\G)-\G)} + \onenorm{\phi} + \onenorm{\delta_h}\Big) \\
	&\quad + C_d\twonorm{p}\Big(\twonorm{\sigma^{3/2}\nabla (P_h(\G)-\G)} + \twonorm{\sigma^{3/2}\phi} + \twonorm{\sigma^{3/2}\delta_h}\Big)\\
	&\leq C\norm{p}_{L^{\infty}(D_2)} + C_d\twonorm{p} \label{eq:pressure_local_I1}.
	\end{align}
	To arrive at this bound, we used \Cref{lemma:sigma_delta} and that\\ $\twonorm{\sigma^{3/2}\phi}\leq \twonorm{\phi}\inftynorm{\sigma^{3/2}}\leq C$.
	Using \cref{eq:def_stoke_fem_equation} and \cref{eq:green_pressure_1} we see for $I_2$
	\begin{align}
	I_2 &= (\nabla(\u - \u_h),\nabla P_h (\G)) = (\nabla(\u - \u_h),\nabla \G) + (\nabla (\u - \u_h),\nabla (P_h(\G)-\G)) \\
	&= -(\Lambda,\nabla \cdot (\u-\u_h)) + (\nabla (\u - \u_h),\nabla (P_h(\G)-\G)) \\
	&= -(\Lambda - r_h(\Lambda),\nabla \cdot (\u-\u_h)) + (\nabla (\u - \u_h),\nabla (P_h(\G)-\G))\\
	&\leq \Big(\norm{\nabla\u}_{L^{\infty}(D^*)}+\norm{\nabla\u_h}_{L^{\infty}(D^*)})(\onenorm{\Lambda - r_h(\Lambda)} + \onenorm{\nabla(P_h(\G)-\G)}\Big)\\
	&\quad + C_d\Big(\twonorm{\nabla(\u-\u_h)})(\twonorm{\sigma^{3/2}(\Lambda - r_h(\Lambda))} + \twonorm{\sigma^{3/2}\nabla(P_h(\G)-\G)}\Big).
	\end{align}
	Here again we use that $\sigma^{-1}$ is bounded by $d$ on $\Omega \backslash D_2$ and choose $D^*$ appropriately such that we can apply \Cref{theorem:W1infty_localization} for the velocity, e.g. $D^* = B(\xx)_{r^*}\cap \Omega$ with $r^*=r + d/2$. Finally $H^1$ stability for $\u_h$ follows by \Cref{proposition:fem_convergence} and we get 
	\begin{equation}
	I_2 \leq C\Big(\norm{\nabla\u}_{L^{\infty}(D_2)}+\norm{p}_{L^{\infty}(D_2)}\Big)+ C_d \Big(\twonorm{\nabla \u} + \twonorm{p}\Big).
	\end{equation}
\end{proof}

\section{Assumptions and main results in two dimensions}
In this section we give a short derivation of the respective local estimates in $L^{\infty}$ and $W^{1,\infty}$ for the two dimensional case. Note that the localization arguments made in the three dimensional case are independent of the dimension apart from the auxiliary estimates. For two dimensions the respective estimates of the regularized Green's functions and the Ritz projection are all available from the literature albeit under slightly different assumptions on the finite element space. 

In the following, we state the required assumptions, the necessary auxiliary results, their references and finally the local estimates. From now on let $\Omega\subset \R^2$, a convex polygonal domain, and consider the two dimensional analogs $\u$, $p$, $\f$ and their finite element discretization as well as the respective two dimensional function and finite element spaces. The basic results and requirements for the continuous problem from \Cref{section:continuousproblem,ss:fem_approximation} still apply, as referenced in these sections.

As stated in \cite{MR2121575}, assume that we have approximation operators \\$P_h \in \mathcal{L}(H^1_0(\Omega)^2;V_h)$ and $r_h \in \mathcal{L}(L^2(\Omega);\bar M_h)$ which fulfill the two dimensional versions of \cref{assumption:stability,assumption:divergence,assumption:inverse_inequality,assumption:approximation} and in addition the following super-approximation properties.
\begin{assumption}[Super-Approximation II]
	\label{assumption:superapproxmation}
	Let $\mu \in [2,3]$, $\v_h \in \V_h$ and $\vpsi = \sigma^{\mu} \v_h$, then
	\begin{equation}
	\twonorm{\sigma^{-\mu/2}\nabla(\vpsi -  P_h(\vpsi))}\leq C \twonorm{\sigma^{\mu/2}\v_h} \quad \forall \v_h \in \V_h
	\end{equation}
	and if $q_h \in \bar M_h$ and $\xi = \sigma^{\mu}q_h$, then
	\begin{equation}
	\twonorm{\sigma^{-\mu/2}(\xi - r_h(\xi))} \leq Ch \twonorm{\sigma^{\mu/2}q_h} \quad \forall q_h \in \bar M_h.
	\end{equation}
\end{assumption}
As in the three dimensional case, this holds for Taylor-Hood finite element spaces, see, e.g. \cite{MR2121575}.
Apart from this, we need to adapt the estimates for $\delta_h$ and $\sigma$. For the two dimensional versions we get
\begin{align}
\norm{\delta_h}_{W^k_q(T_{\x_0})}&\leq C h^{-k-2(1-1/q)}, \quad 1 \leq q \leq \infty, k=0, 1, \dots, \quad \nu>0 \quad \text{and} \label{eq:deltah_est_2d} \\
\twonorm{\sigma^{\nu}\nabla_k \delta_h} &\leq 2^{\nu/2}C\kappa^{\nu}h^{\nu-k-1} \quad k=0,1.
\end{align}
Let $(\g_1,\lambda_1)$ and $(\g_0,\lambda_0)$ denote the two dimensional regularized Green's functions, defined as in \Cref{section:proof_of_main_theorems} but for two dimensions. Then we get the following convergence estimates for their discrete counterparts. The estimates needed when deriving $W^{1,\infty}$ velocity estimates,
\begin{equation}
\onenorm{\nabla (\g_1-\g_{1,h})} \leq C, \qquad
\twonorm{\sigma\nabla (\g_1-\g_{1,h})} \leq C
\end{equation}
follow from \cite[Theorem 8.1]{MR2121575} using \cref{eq:lower_bound_sigma} and similarly for the pressure estimates where we need
\begin{align}
\onenorm{\nabla(P_h(\G)-\G)} + \onenorm{r_h(\Lambda)-\Lambda} &\leq C,\\
\twonorm{\sigma\nabla(P_h(\G)-\G)} + \twonorm{\sigma(r_h(\Lambda)-\Lambda)} &\leq C
\end{align}
which can be found in \cite[p. 328]{MR2121575}.
In the $L^{\infty}$ case for the velocity we get
\begin{equation}
\onenorm{\nabla(\g_0-\g_{0,h})} \leq Ch \lnhh,\qquad
\twonorm{\sigma\nabla(\g_0-\g_{0,h})} \leq Ch \lnhh^{1/2}
\end{equation}
from \cite[Theorem 4.1, Proof of Theorem 4.2]{MR935076}. The equivalent version of \Cref{lemma:lambda0} is given by \cite[Lemma 3.1]{MR935076}. Finally the estimate for the Ritz projection $R_h$ in two dimensions
\begin{equation}
\inftynorm{R_h \vec z} \leq C\lnhh \inftynorm{\vec z}
\end{equation}
is given in \cite{MR551291}. Note that the local maximum norm estimates for $L^{\infty}$ from \cite{MR3614014} hold as well in two dimensions.
Thus, using the same techniques as in \Cref{section:proof_of_main_theorems} we get the following theorems for $\Omega \subset \R^2$.
\begin{theorem}[Interior $W^{1,\infty}$ estimate for the velocity and $L^{\infty}$ estimate for the pressure]
	\label{theorem:w1_localization_2D}
	Under the assumptions above, $\Omega_1 \subset \Omega_2 \subset \Omega$ with $dist(\bar \Omega_1, \partial \Omega_2)\geq d \geq \bar \kappa h$ and if $(\u,p)\in (W^{1,\infty}(\Omega_2)^2 \times L^{\infty}(\Omega_2))\cap (H^1_0(\Omega)^2 \times L^2_0(\Omega))$ is the solution to \cref{eq:def_stoke_1,eq:def_stoke_2,eq:def_stoke_3}, then it holds for $(\u_h,p_h)$ the solution to \cref{eq:def_stoke_fem_equation}:
	\begin{multline}
	\norm{\nabla \u_h}_{L^{\infty}(\Omega_1)} +\norm{p_h}_{L^{\infty}(\Omega_1)}\nonumber\\
	\leq C\Big(\norm{\nabla\u}_{L^{\infty}(\Omega_2)} + \norm{p}_{L^{\infty}(\Omega_2)} \Big) 
	+C_d\Big(\twonorm{\nabla\u} +\twonorm{p}\Big).
	\end{multline}
	Here, the constant $C_d$ depends on the distance to $\Omega_1$ from $\partial \Omega_2$.
\end{theorem}
\begin{theorem}[Interior $L^{\infty}$ error estimate for the velocity]
	Under the assumptions above, $\Omega_1 \subset \Omega_2 \subset \Omega$ with $dist(\bar\Omega_1, \partial \Omega_2)\geq d \geq \bar \kappa h$ and if $(\u,p)\in (L^{\infty}(\Omega_2)^2 \times L^{\infty}(\Omega_2))\cap (H_0^1(\Omega)^2 \times L_0^2(\Omega))$ is the solution to \cref{eq:def_stoke_1,eq:def_stoke_2,eq:def_stoke_3}, then it holds for $(\u_h,p_h)$ the solution to \cref{eq:def_stoke_fem_equation}:
	\begin{multline}
	\norm{\u_h}_{L^{\infty}(\Omega_1)}\nonumber
	\leq C\lnhh\Big(\lnhh \norm{\u}_{L^{\infty} (\Omega_2)} + h\norm{p}_{L^{\infty} (\Omega_2)}\Big) \\
	+ C_d\lnhh^{1/2}\Big( h \norm{\u}_{H^1(\Omega)}+ \twonorm{\u} + h \twonorm{p} \Big).
	\end{multline}
	Here, the constant $C_d$ depends on the distance to $\Omega_1$ from $\partial \Omega_2$.
\end{theorem}

\bibliographystyle{siamplain}
\bibliography{lit}
\end{document}